\documentclass[11pt]{amsart}
\usepackage{amsfonts}
\usepackage{amsmath}
\usepackage{amsthm}
\usepackage{url}
\usepackage[all]{xy}
\usepackage{latexsym}
\usepackage{amssymb}
\usepackage[cp850]{inputenc}
\usepackage{amsfonts}
\usepackage{graphicx}

\usepackage[abs]{overpic}

\usepackage[usenames]{color}

\newtheorem{sat}{Theorem}[section]		
\newtheorem{lem}[sat]{Lemma}
\newtheorem*{claim}{Claim}
\newtheorem{kor}[sat]{Corollary}			
\newtheorem{prop}[sat]{Proposition}

\newtheorem*{defi*}{Definition}			
\newtheorem*{bei*}{Example}
\newtheorem*{sat*}{Theorem}				
\newtheorem*{kor*}{Corollary}
\newtheorem*{rmk*}{Remark}				
\newtheorem{quest}{Question}	
\newtheorem*{quest*}{Question}	
\newtheorem{fact}{Fact}	

\let\ssection=\section
\renewcommand{\section}{\setcounter{equation}{0}\ssection}

\newtheorem*{namedtheorem}{\theoremname}
\newcommand{\theoremname}{testing}
\newenvironment{named}[1]{\renewcommand{\theoremname}{#1}\begin{namedtheorem}}{\end{namedtheorem}}

\theoremstyle{remark}
\newtheorem*{bem}{Remark}
\newtheorem{bei}{Example}

\newtheorem*{namedtheoremr}{\theoremnamer}
\newcommand{\theoremnamer}{testing}

\newcommand{\BR}{\mathbb R}			
\newcommand{\BN}{\mathbb N}			
\newcommand{\BS}{\mathbb S}

\newcommand{\BB}{\mathbb B}

		\newcommand{\CB}{\mathcal B}

		\newcommand{\CN}{\mathcal N}

\newcommand{\D}{\partial}
\newcommand{\DD}{\nabla}

\DeclareMathOperator{\vol}{vol}		

\DeclareMathOperator{\inj}{inj}
\DeclareMathOperator{\diam}{diam}

\newcommand{\comment}[1]{}

\DeclareMathOperator{\SO}{SO}

\DeclareMathOperator{\const}{const}

\DeclareMathOperator{\Sing}{Sing}
\DeclareMathOperator{\capa}{cap}
\DeclareMathOperator{\Const}{const}

\newcommand{\norm}[1]{\lVert #1 \rVert}

\newcommand{\dmu}{\;\mathrm{d}\mu}

\newcommand{\QR}{\mathrm{QR}}
\newcommand{\Strata}{\mathrm{Strata}}

\begin{document}

\title[]{Bubbling of quasiregular maps}
\author{Pekka Pankka}
\address{Department of Mathematics and Statistics, P.O. Box 68 (Pietari Kalmin katu 5), FI-00014 University of Helsinki, Finland}
\email{pekka.pankka@helsinki.fi}
\author{Juan Souto}
\address{Univ Rennes, CNRS, IRMAR - UMR 6625, F-35000 Rennes, France}
\email{jsoutoc@gmail.fr}
\date{\today}

\thanks{P.P. was partially supported by the Academy of Finland project \#297258.}
\thanks{Both authors were partially supported by the Projet international de coop\'eration scietifique PICS 7734.}

\begin{abstract}
We give a version of Gromov's compactess theorem for pseudoholomorphic curves in the case of quasiregular mappings between closed manifolds. More precisely we show that, given $K\ge 1$ and $D\ge 1$, any sequence $(f_n \colon M \to N)$ of $K$-quasiregular mappings of degree $D$ between closed Riemannian $d$-manifolds has a subsequence which converges to a $K$-quasiregular mapping $f\colon X\to N$ of degree $D$ on a nodal $d$-manifold $X$.  
\end{abstract}

\maketitle

\section{Introduction}

A continuous map $f\colon M\to N$ between closed, oriented, $d$-dimensional Riemannian manifolds is $K$-\emph{quasiregular} for some $K\ge 1$ if it belongs to the Sobolev space $W^{1,d}_\mathrm{loc}(M,N)$ and satisfies 
\begin{equation}
\label{eq:K}
\norm{df}^d \le K\cdot \det(df)
\end{equation}
almost everywhere in $M$. A map is {\em quasiregular} if it is $K$-quasiregular for some $K$ and \emph{quasiconformal} if it is a quasiregular homeomorphism. 

For $K\ge 1$ and $D\ge 0$, we denote by $\QR_{K,D}(M,N)$ the space of all $K$-quasiregular maps $M\to N$ of degree exactly $D$, endowed with the topology of uniform convergence. The main result of this paper is a version for $\QR_{K,D}(M,N)$ of Gromov's compactness theorem for pseudoholomorphic curves \cite{Gromov}. Informally, Theorem \ref{main} below asserts the following:

\begin{quote}
{\em Let $M$ and $N$ be closed Riemannian $d$-manifolds and fix $K$ and $D$. Every sequence  $(f_n:M\to N)$ of $K$-quasiregular maps of degree $D$ converges, up to passing to a subsequence, to a $K$-quasiregular map $f:X\to N$, also of degree $D$, on some nodal Riemannian $d$-manifold $X$.}
\end{quote}

In order to be able to state this theorem precisely we need to clarify the used terminology. A (pure) {\em nodal $d$-manifold} is a topological space $X$, which in the complement of a finite set $\Sing(X)$ of singularities, is a $d$-manifold,  each point in $\Sing(X)$ separates $X$ into two components, and the local structure of $X$ near $p\in\Sing(X)$ is modelled on that of two $d$-dimensional planes in general position in $\BR^{2d}$. 

For a nodal manifold $X$, the closure of each component of $X\setminus\Sing(X)$ is called a {\em stratum} of $X$ and each stratum is a manifold in its own right. A \emph{Riemannian nodal manifold} is a nodal manifold, where each stratum is endowed with a Riemannian metric. A continuous map $f:X\to N$ from a nodal $d$-manifold to a Riemannian $d$-manifold is {\em $K$-quasiregular} if its restriction to each stratum is $K$-quasiregular. We say that a sequence of maps $(f_n:M\to N)$ between manifolds {\em converges to a map $f:X\to N$} on some nodal manifold $X$ if there is a sequence ($\pi_n:M\to X$) of pinching maps such that the maps 
\[
f_n\circ\pi_n^{-1}:X\setminus\Sing(X)\to N
\]
converge to $f$. Here a map $M\to X$ is a {\em pinching map} if it is a homeomorphism in the complement of a finite set of disjoint embedded spheres which are being pinched to the singular points in $\Sing(X)$. Note that $X$ has a stratum homeomorphic to $M$. This one is the {\em main stratum} and all the other are the {\em bubbles} of $X$. See Section \ref{sec-nodal} for a more formal discussion of nodal manifolds, pinching maps, (tight) convergence, etc.

Having these notions in place, we may state our main theorem.

\begin{sat}\label{main}
Let $M$ and $N$ be closed and oriented Riemannian $d$-manifolds for $d\ge 2$, and let $K\ge 1$ and $D\ge 1$. Then each sequence $(f_n:M\to N)$ of $K$-quasiregular maps of degree $D$ has a subsequence $(f_{n_k})$, which converges tightly to a $K$-quasiregular map $f:X\to N$ on some smooth pure nodal Riemannian $d$-manifold $X$ satisfying 
\begin{equation}
\label{eq:degree}
D=\sum_{V\in\Strata(X)}\deg(f\vert_V).
\end{equation}
Furthermore, the main stratum of $X$ is diffeomorphic to $M$ and all bubbles are standard spheres $\BS^d$.
\end{sat}

A few comments before discussing some applications of this theorem:
\smallskip

\noindent{\bf 1.} An interesting feature of Theorem \ref{main} is that it is {\em global}. By this we mean that, contrary to what happens in other compactness results for quasiregular maps, we suffer no loss of degree, c.f. \eqref{eq:degree}. With notation as in Theorem \ref{main}, this can be explained as follows. The chosen subsequence $(f_{n_k})$ converges locally uniformly in the complement of the finite set $P=M\cap\Sing(X)$ of singular points contained in the main stratum $M$ of $X$. In the terminology of Pang, Nevo, and Zalcman \cite{Pang-Nevo-Zalcman}, the restriction of $f:X\to N$ to $M$ is the quasinormal limit of the sequence $(f_{n_k})$. When just considering convergence on $M\setminus P$ we encounter loss of degree. This is due to the fact that the convergence of $(f_{n_k})$ at the poles is not locally uniform. We solve this problem by rescaling the neighborhoods of the poles, starting to grow bubbles, i.e. the other strata of $X$. In this way the poles of the sequence turn into singular points on the main stratum of $X$. The strata of $X$, not touching $M$, are formed similarly, that is, by further rescaling already rescaled maps. 
\smallskip

\noindent{\bf 2.}
Compensating the loss of energy, or in our case loss of degree, by rescaling to make bubbles appear is definitively not a new idea. It is central to Gromov's proof of his compactness theorem. And in fact, it also plays a key role in the earlier paper of Sacks--Uhlenbeck \cite{Sacks-Uhlenbeck-1981} guaranteeing the existence of minimal spheres in certain 3-manifolds.
\smallskip

\noindent{\bf 3.}
In Theorem \ref{main} the degree of the maps is fixed. The problem of understanding the asymptotic behaviour of sequences of $K$-quasiregular maps whose degree is tending to $\infty$ arises naturally in dynamics, but not only there (see for example \cite{Bonk-Heinonen,Kangasniemi,Martin-Mayer, Martin-Mayer-Peltonen, Prywes}). This problem is however very different from the one treated in this paper. While Theorem \ref{main} and its consequences are of global nature, the relevant tool to study the renormalization of $K$-quasiregular maps of growing degree is the local, or even infinitesimal, Miniowitz--Zalcman theorem (see Miniowitz \cite{Miniowitz} or Iwaniec--Martin \cite[Theorem 19.9.3]{Iwaniec-Martin-book}). 
\smallskip

\noindent{\bf 4.} As stated, the nodal manifold $X$ in Theorem \ref{main} is not unique. Even if one factors in the fact that it involves choosing subsequences. The point is that we have not introduced any version of the notion of {\em stablity} used for pseudo-holomorphic curves. It would be however easy to do so, and we leave this to the interested reader.
\medskip

We come now to the applications of Theorem \ref{main}.

\subsection*{Conditions ensuring that $\QR_{K,D}(M,N)$ is compact}

In 1964 Lichnerowitz \cite{Lichnerowicz} conjectured that the sphere is the only closed Riemannian manifold which has a non-compact group of conformal automorphisms. A few years later Lelong-Ferrand, incidentally the first woman admitted to the \'Ecole Normale Sup\'erieure, noticed that the quasiconformal world is the proper environment for Lichnerowitz's conjecture. In fact, to prove the conjecture, she proved that if we fix $K\ge 1$ then the following holds: {\em If $M$ and $N$ are closed oriented Riemannian manifolds of dimension $d$, then the space $\QR_{K,1}(M,N)$ of $K$-quasiconformal homeomorphisms between them is compact unless there is a quasiconformal homeomorphism $\BS^d\to N$}. We refer to Lelong-Ferrand's ICM paper \cite{Lelong-Ferrand-ICM} for references and for a general discussion of Lichnerowicz's conjecture. Anyways, as a first application of Theorem \ref{main} we extend Lelong-Ferrand's result from degree $1$ to arbitrary degree. 

\begin{sat}\label{sat-compact}
Let $M$ and $N$ be closed and oriented Riemannian $d$-manifolds for $d\ge 2$. Then the following statements are equivalent:
\begin{enumerate}
\item The space $\QR_{K,D}(M,N)$ is not compact for some $K\ge 1$ and $D\ge 1$.
\item There exists a non-constant quasiregular map $\BS^d\to N$.
\end{enumerate}
Furthermore, if $\QR_{K,D}(M,N)$ is non-compact, then the map in (2) has degree at most $D$. 
\end{sat}

It follows from Theorem \ref{sat-compact} that the topology of $N$ is strictly restricted if $\QR_{K,D}(M,N)$ is non-compact. 

\begin{kor}\label{kor-bubble-rational}
Let $M$ and $N$ be closed and oriented Riemannian $d$-manifolds and suppose that $\QR_{K,D}(M,N)$ is not compact for some $K\ge 1$ and $D\ge 1$. Then $N$ is a rational homology sphere with a finite fundamental group and non-trivial $\pi_d(N)$. If in addition $d\le 4$ then $N$ is covered by $\BS^d$.
\end{kor}

Returning for a moment to Gromov's compactness theorem, recall that one of its consequences is -- if for example the involved maps are $\pi_1$-injective -- that {\em degenerations of pseudo-holomorphic curves can only occur if the target has bubbles}. Since in our setting both target and domain have the same dimension $d$, we get that the relevant bubbles are also $d$-dimensional. This is reflected in Corollary \ref{kor-bubble-rational} by the non-trivially of $\pi_d(N)$ when $\QR_{K,D}(M,N)$ is not compact.

Continuing with our discussion of Theorem \ref{sat-compact}, note that the following example shows that Corollary \ref{kor-bubble-rational} is essentially optimal. 

\begin{bei}
Given $D\ge 2$, let $N$ be the quotient of $\BS^3$ under the action of the subgroup of $\SO(4)$ generated by $(\zeta,\eta)\mapsto(e^{\frac{2\pi i}D}\zeta,e^{\frac{2\pi i}D}\eta)$. Then $\QR_{K,D}(\BS^3,N)$ is non-compact: the covering map $\BS^3\to N$ is conformal and the group of conformal automorphisms of $\BS^3$ is not compact. Consistently with Corollary \ref{kor-bubble-rational}, the lens space $N$ is a rational homology sphere with finite fundamental group but it is neither an integral homology sphere, nor has trivial fundamental group. 
\end{bei}

At this point the reader could be thinking that this example is not fair because $N$ is a quotient of the sphere. The sphere is indeed the only example known to us of a simply connected manifold which is quasiregularly dominated by the sphere. 

\begin{quest}\label{quest1}
Let $N$ be a closed and simply connected Riemannian $d$-manifold and suppose that there is a non-constant quasiregular map $\BS^d\to N$. Is $N$ homeomorphic to $\BS^d$?
\end{quest}

In the course of the proof of Corollary \ref{kor-bubble-rational} we note that a positive answer to this question follows in low-dimensions from the classification of surfaces, the Poincar\'e conjecture, and Freedman's classification of simply connected 4-manifolds. It would be great to give an answer to Question \ref{quest1} from a quasi-conformal point of view, that is without using deep topological classification results, and we encourage the reader to attempt it. Once this is said, we should also warn him or her that having such a positive answer would in turn give a proof of the 3-dimensional Poincar\'e conjecture. We sketch the argument at the end of Section \ref{sec:bubbling-applications}.

\subsection*{Finiteness of homotopy classes}
Taking the result of Theorem \ref{sat-compact} to another direction, note that the set of all $K$-quasiregular maps $M\to N$ of degree $D$ represents only finitely many homotopy classes unless there is a non-constant quasiregular map $\BS^d\to N$. 
We deduce from our main theorem that finiteness of homotopy classes holds in full generality.

\begin{sat}\label{sat-finite homotopy classes}
Let $M$ and $N$ be closed and oriented Riemannian $d$-manifolds for some $d\ge 2$. Then, for each $K\ge 1$ and $D\ge 1$, there exist only finitely many homotopy classes of maps $M\to N$ of degree $D$, which admit a $K$-quasiregular representative.
\end{sat}

Theorem \ref{sat-finite homotopy classes} raises the obvious question, reminiscent of Gromov's result \cite[Theorem 2.18]{Gromov-book} on counting homotopy classes of maps with a given Lipschitz constant:

\begin{quest}
Let $M$ and $N$ be closed Riemannian manifolds and fix $D\ge 1$. What is the growth rate of the number of the homotopy classes represented by maps in $\QR_{K,D}(M,N)$ when $K\to\infty$?
\end{quest}

Note that when investigating this problem one might have to deal with the difficulty of determining when is a given map homotopic to a quasiregular map. In dimensions $d=2,3$, this problem can be essentially solved using results of Edmonds \cite{Edmonds-deg3, Edmonds-deg2} whereas the problem seems much more subtle in higher dimensions. It might thus be reasonable to concentrate in the case $D=1$, that is, that all involved quasiregular maps are homeomorphisms. In fact, already in this case, and even in dimension 3, one finds really very interesting cases:

\begin{quest}
Let $M=(\BS^1\times\BS^2)\#(\BS^1\times\BS^2)\#(\BS^1\times\BS^2)$ be the connected sum of three copies of $\BS^1\times\BS^2$. What is the growth rate of the number of the homotopy classes represented by maps in $\QR_{K,1}(M,M)$ when $K\to\infty$?
\end{quest}

Incidentally, recall that Bridson, Hinkkanen, and Martin \cite{Bridson-Hinkkanen-Martin} noted that every non-constant quasiregular self-map of the connected sum of three copies of $\BS^1\times\BS^2$ (or more generally of a manifold with non-elementary torsion free Gromov hyperbolic fundamental group) has degree $1$.

\subsection*{Plan of the paper}

To prove Theorem \ref{main} we will follow the basic idea behind the proof of Gromov's compactness theorem for pseudo-holomorphic curves. What is underpinning the whole strategy is the similarity between the more geometric aspects of the theory of holomorphic functions of one variable and that of quasiregular maps. We recall the relevant facts about quasiregular maps in Section \ref{sec:qr} focusing on analogues of Montel's theorem, and also noting that, in a sense to be made precise below, a certain "energy" (the integral of the Jacobian) associated to a quasiregular map between closed manifolds is "quantized" in the sense that it takes only finitely many values.

In Section \ref{sec:poles} we prove, as the first step towards the proof of Theorem \ref{main}, the quasinormality of $\QR_{K,D}(M,N)$ for fixed $K\ge 1$ and $D\ge 1$. More precisely we prove that, up to passing to a subsequence, any sequence in $\QR_{K,D}(M,N)$ has a subsequence which converges locally uniformly in the complement of a finite collection of "poles", and that whenever a pole appears a definite amount of energy is lost.

We discuss in Section \ref{sec-nodal} what is a nodal manifold and what does it mean for a sequence of maps to converge to a map on a nodal manifold. In the setting of Theorem \ref{main}, the nodes of the relevant nodal manifold will arise at the poles. This is shown in Section \ref{sec:proof-main} where Theorem \ref{main} is also proved. 

We conclude the paper with Section \ref{sec:bubbling-applications}, where we prove Theorem \ref{sat-compact}, Corollary \ref{kor-bubble-rational}, and Theorem \ref{sat-finite homotopy classes}.

\subsection*{Acknowledgements}
The authors thank Bourbon D'Arsel for hospitality and all the coffee.

\section{Preliminaries on quasiregular maps}
\label{sec:qr}
In this section we recall few facts on quasiregular maps, referring the reader to \cite{Rickman-book,Vaisala-book} and \cite{Bojarski-Iwaniec} for details. We also address the reader to V\"ais\"al\"a's survey \cite{VaisalaICM} of classical results in this field. All these sources are concerned with local results, that is, with quasiregular maps between open sets of Euclidean spaces. Global results about quasiregular maps between Riemannian manifolds, such as those mentioned below, follow easily from the local ones when one considers first the problem in charts. We leave details to the reader.

\subsection*{Notation} 
Let $M$ be a Riemannian manifold. We denote by $B^M(x,r)$ and $S^M(x,r)$ the \emph{ball} and \emph{sphere} of radius $r$ around $x$, respectively. However, if $M$ is the Euclidean space $\BR^d$, then we use $\BB(x,r)=B^{\BR^d}(x,r)$ and $\BS(x,r)=S^{\BR^d}(x,r)$ instead. If $\mu$ is a finite measure on a manifold $M$ then we let $\Vert\mu\Vert=\mu(M)$ be the total measure. Recall that all manifolds are assumed to be $d$-dimensional, smooth, and oriented. Finally, when working in manifolds we will often abuse terminology and identify the domains of the charts, the targets of the charts, and the actual charts.

\subsection*{Regularity and equicontinuity}

As already defined in the introduction, a continuous mapping $f\colon M\to N$ between two $d$-dimensional Riemannian manifolds is \emph{quasiregular} if it belongs to the Sobolev space $W^{1,n}_{\mathrm{loc}}(M,N)$ and there exists a constant $K\ge 1$ for which \eqref{eq:K} above is satisfied for almost every $x\in M$. 

By results of Reshetnyak, the regularity properties of quasiregular maps go much beyond a priori assumptions. We record this list of properties as facts and refer the interested reader to \cite[Chapter I]{Rickman-book} for a more complete list. 

The first fact is a combination of Reshetnyak's results:

\begin{itemize}
\item[(QR1)] {\em Quasiregular maps between Riemannian manifolds are either discrete and open or constant. Moreover, they are almost everywhere differentiable and, as long as they are non-constant, they are also absolutely continuous in measure.} 
\end{itemize}

More precisely, the Jacobian $\det df$ of a non-constant quasiregular map $f:M \to N$ is positive almost everywhere and for any $E\subset M$ we have that $\vol_N(fE)=0$ if and only if $\vol_M(E)=0$.

\begin{bem}
Interestingly, the prominent property of quasiregular maps being discrete and open plays a very subtle role in what follows, whereas the analytic condition on being absolutely continuous in measure is the star of the show.
\end{bem}

It follows from (QR1) that the pull-back under $f$ of the volume form $\vol_N$ of $N$ yields a measure 
\[
\mu_f=f^*\vol_N
\]
on $M$ in the Lebesgue class and $\mu_f$ is the measure with derivative
\[
\frac{d\mu_f}{d\vol_M}(x)=\det(df_x) \quad \text{a.e. in }M.
\]
We refer to the measure $\mu_f$ as the {\em energy density} of $f$ and to the total measure $\Vert\mu_f\Vert$ as the {\em total energy} of $f$. Note that, $\norm{\mu_f}=0$ if and only if $f$ is a constant map.

\begin{bem}
It may be bit curious to call $\mu_f$ the energy density of $f$. Especially, since typically in the quasiconformal literature the functional 
\[
f\mapsto \int_M \norm{df}^d d\vol_M 
\]
is called the $d$-energy of $f$. The reason for this terminology stems from the fact that, in our considerations, total energy $\norm{\mu_f}$ admits a version of quantization, which we discuss shortly. Note that, by \eqref{eq:K}, the $d$-energy and $\norm{\mu_f}$ are comparable with a constant depending only on $K$ and $d$.
\end{bem}

Before going further, note that the energy density behaves well when the map is precomposed with an orientation preserving diffeomorphism, and more generally with a quasiconformal map:

\begin{itemize}
\item[(QR2)] {\em We have $\mu_{f\circ\phi}=\phi^{-1}_*\mu_f$ for $f:M\to N$ quasiregular and $\phi:M'\to M$ quasiconformal. }
\end{itemize}

As it is the case for holomorphic maps, local uniform limits of quasiregular maps are quasiregular \cite[Theorem VI.8.6]{Rickman-book}. Since Jacobian determinants of locally uniformly converging quasiregular maps converge weakly \cite[Lemma VI.8.8]{Rickman-book}, we have the following fact.

\begin{itemize}
\item[(QR3)] {\em Let $M$ and $N$ be connected and orientable Riemannian $d$-manifolds. Suppose that a sequence $(f_n:M\to N)$ of $K$-quasiregular maps converges uniformly on compact subsets to some map $f:M\to N$. Then $f$ is also $K$-quasiregular and $\mu_f=\lim_{n\to \infty} \mu_{f_n}$ in the weak-$\ast$ topology, that is, 
\[
\int_M \varphi \dmu_{f_n} \to \int_M \varphi \dmu_f
\]
for all $\varphi \in C_0(M)$ as $n\to \infty$.}
\end{itemize}

Returning to the regularity of quasiregular maps, Meyers and Elcrat proved in \cite{Meyers-Elcrat} that the quasiregularity condition \eqref{eq:K} implies that quasiregular maps belong actually to $W^{1,p}_{\mathrm{loc}}(M,N)$ for some $p=p(d,K)>n$. As such, quasiregular maps $f:M\to N$ are H\"older continuous. In fact, if the domain and target are open subsets of Euclidean space, then the H\"older exponent and the H\"older coefficient are bounded in terms of the energy $\Vert\mu_f\Vert$ of $f$; see Bojarski--Iwaniec \cite[Theorem 5.2]{Bojarski-Iwaniec} for a proof of the next fact.

\begin{itemize}
\item[(QR4)] {\em Given $\Omega\subset\BR^d$ open, $F\subset\Omega$ compact, and $K\ge 1$, there are $C>0$ and $\alpha=\alpha(d,K)>0$ with
\[
\vert f(x)-f(y)\vert\le C\cdot \Vert\mu_f\Vert^{\frac 1d}\cdot\vert x-y\vert^\alpha
\]
for any $K$-quasiregular map $f:\Omega\to\BR^d$ and any $x,y\in F$.}
\end{itemize}

It is worth noting that the statement of (QR4) fails in the global setting. Consider, for example, the non-compact group of conformal automorphisms of the sphere $\BS^d$. However, also in the global setting one finds oneself often working with equicontinuous families of quasiregular maps.

\begin{itemize}
\item[(QR5)] {\em Let $M$ and $N$ be connected and oriented Riemannian $d$-manifolds, and suppose that $N$ is closed. Let $0<c<\vol_N(N)$. Then, for each $K\ge 1$, the family 
\[
\left\{f:M\to N\middle\vert\ K\text{-quasiregular map with }\Vert\mu_f\Vert\le c
\right\}
\]
is equicontinuous.}
\end{itemize}

This version of Montel's theorem is an easy consequence of a theorem of Martio, Rickman, and V\"ais\"al\"a (see \cite[Corollary III.2.7]{Rickman-book}) which states that \emph{a family of quasiregular maps omitting a set of positive capacity is equicontinuous}. We explain how (QR5) follows. First we need some terminology.

If $A\subset N$ is open and $C\subset A$ is compact, then the {\em $p$-capacity} of $(A,C)$ is defined to be
\[
\capa_p(A,C)=\inf_u\int_A\vert\DD u\vert^pd\vol,
\]
where the infimum is taken over all smooth functions with compact support contained in $A$ and which satisfy $u\vert_C\equiv 1$. A compact set $C\subset N$ is said to \emph{have positive capacity} if it has an open neighbourhood $A\ne N$ with $\capa(A,C)>0$. Note that compact sets $C\ne N$ of positive volume on a connected manifold have positive capacity. 

Returning now to the proof of (QR5) suppose that $(f_n:M\to N)$ is a sequence of $K$-quasiregular maps with $\Vert\mu_{f_n}\Vert\le c<\vol(N)$ for all $n\in\BN$, and note that this implies that 
\[
\vol(f_n(M))\le\Vert\mu_{f_n}\Vert\le c<\vol(N)
\]
for all $n$. It follows thus from the Borel-Cantelli lemma that there are a compact set $C\subset N$ with positive volume and a subsequence $(f_{n_i})$ of $(f_n)$ with $f_{n_i}(M)\subset M\setminus C$ for all $i$. Now, any family of $K$-quasiregular maps whose images avoid a set of positive capacity is equicontinuous; see  \cite[Corollary III.2.7]{Rickman-book}. The claim of (QR5) follows.\qed

\subsection*{Maps with small energy}
In the course of this paper we will need to deal a few times with maps with small energy. We start with a simple consequence of (QR4) and (QR5).

\begin{lem}\label{lem-neck}
Let $N$ be a closed and oriented Riemannian $d$-manifold. For all $\epsilon>0$ and $K\ge 1$ there is $C>0$ having the following property: 
Let $f:(-2,2)\times\BS^{d-1}\to N$ be a $K$-quasiregular map satisfying $\Vert\mu_f\Vert<C$. Then
\[
\diam_N\left(f((-1,1)\times\BS^{d-1})\right)<\epsilon
\]
and there exists $r\in (-1,1)$ for which
\[
\int_{\{r\}\times\BS^{d-1}}\Vert df\Vert^{d-1} d\vol_{\BS^{d-1}}<\epsilon.
\]
\end{lem}
\begin{proof}
First we note that there are $K_0\ge 1$ and a finite number (two suffices) of $K_0$-quasiconformal embeddings
$$\phi_1,\dots,\phi_k:\BB(0,2)\to (-2,2)\times\BS^{d-1}$$
with 
$$(-2,2)\times\BS^{d-1}=\bigcup_{i=1}^k\phi_i(\BB(0,2))\text{ and }(-1,1)\times\BS^{d-1}\subset\bigcup_{i=1}^k\phi_i(\BB(0,1)).$$
It follows that
\begin{equation}\label{eq-trump is a piece of shit1}
\diam_N(f((-1,1)\times\BS^{d-1}))\le\sum_{i=1}^k\diam_N\left((f\circ\phi_i)(\BB(0,1))\right).
\end{equation}
Since all the maps $f\circ\phi_i$ are $KK_0$-quasiregular with $\Vert\mu_{f\circ\phi_i}\Vert\le\Vert\mu_f\Vert\le C$, we get from (QR4) and (QR5) that 
\begin{equation}\label{eq-trump is a piece of shit2}
\diam_N((f\circ\phi_i)(\BB(0,1))<\frac 1k\epsilon
\end{equation}
holds as long as $C$ is small enough. Indeed, by (QR5), there exists $r_0>0$ for which the image $(f\circ \phi_i)(\BB(x,r_0))$ is contained in a $2$-bilipschitz chart on $N$ for each $x$. Thus, by (QR4), $\diam (f\circ \phi_i)(\BB(x,r_0)) < C'\mu_{f\circ\phi_i}(\BB(x,r_0))^{1/d}r_0^\alpha$ for each $x$, where $C'$ and $\alpha$ do not depend on $x$. Fix now $x\in \BB(0,1)$ and let $B_1,\ldots, B_k$ be a collection of closed balls of radius $r_0$, having pair-wise disjoint interiors, covering the segment $[0,x]$; note that, since balls $B_i$ are almost disjoint, we have $k\le 2/r_0$, say. Then 
\begin{align*}
|f(\phi_i(0)) - f(\phi_i(x))| &\le \sum_{i=1}^k \diam fB_i \le \sum_{i=1}^k C' \mu_{f\circ \phi_i}(B_i)^{1/d} r_0^\alpha \\
&\le C' \left(\sum_{i=1}^k \mu_{f\circ \phi_i}(B_i) \right)^{1/d} \left( \sum_{i=1}^k r_0^{\alpha \frac{d}{d-1}}\right)^{\frac{d-1}{d}} \\
&\le C'' r_0^{\alpha-1+1/d} \norm{\mu_{f\circ \phi_i}} \le C'' \norm{\mu_{f\circ \phi_i}},
\end{align*}
where $C''$ depends on $C'$ and $r_0$. 

The first claim follows immediately from \eqref{eq-trump is a piece of shit1} and \eqref{eq-trump is a piece of shit2}. 

For the second claim, we observe that the quasiregularity condition implies that
$$\int_{(-1,1)\times\BS^{d-1}}\Vert df\Vert^dd\vol\le K\cdot \mu_f((-1,1)\times\BS^{d-1})\le K\cdot C.$$
Setting $U=(-1,1)\times\BS^{d-1}$, we get from H\"older's inequality that
$$\int_U\Vert df\Vert^{d-1}d\vol\le\left(\int_U\Vert df\Vert^dd\vol\right)^{\frac{d-1}d}\cdot\vol(U)^{\frac 1d}\le\Const\cdot\ C^{\frac{d-1}d}.$$
This means that
$$\int_{-1}^{1}\left(\int_{\{r\}\times\BS^{d-1}}\Vert df\Vert^{d-1}d\vol_{\BS^{d-1}}\right)dr\le\const\cdot\ C^{\frac{d-1}d},$$
which implies that 
$$\inf_{r\in(-1,1)}\int_{\{r\}\times\BS^{d-1}}\Vert df\Vert^{d-1}d\vol_{\BS^{d-1}}\le\Const\cdot\ C^{\frac{d-1}d}.$$
The claim follows.
\end{proof}

Lemma \ref{lem-neck} deals only with annuli of some definite size. We prove next that the images under quasiregular maps of arbitrarily large cylinders have small diameter as long as the energy is small enough. 

\begin{prop}\label{lem-neck2}
For every closed and oriented Riemannian $d$-manifold $N$ and every $K\ge 1$ and $\epsilon>0$, there is $\omega>0$ such that, for every $\Lambda>0$, the following holds: If $f:\left(-(\Lambda+1),\Lambda+1\right)\times\BS^{d-1}\to N$ is a $K$-quasiregular map with $\Vert\mu_f\Vert\le\omega$ then $\diam(f([-\Lambda,\Lambda]\times\BS^{d-1}))\le\epsilon$.
\end{prop}
\begin{proof}
Note first that, since $N$ is a closed manifold, we may assume that $\epsilon$ is such that every $10\epsilon$-ball in $N$ is $2$-bilipschitz to the round ball $\BB(0,10\epsilon)$. Now, by Lemma \ref{lem-neck} there exists $C>0$ such that whenever 
\[
f':(-2,2)\times\BS^{d-1}\to N
\]
is a $2K$-quasiregular map satisfying $\Vert\mu_{f'}\Vert<C$, then 
\begin{equation}\label{eq-abc}
\diam_N(f'([-1,1]\times\BS^{d-1}))<\frac 1{40}\epsilon.
\end{equation}

Fix now a point $p\in\{-\Lambda\}\times\BS^{d-1}$. By \eqref{eq-abc}, we have $f(\{-\Lambda\}\times \BS^{d-1}) \subset B^N(f(p),\frac{1}{40}\epsilon)$. We claim that $f([-\Lambda,\Lambda]\times\BS^{d-1})\subset B^N(f(p),\frac 9{20}\epsilon)$. Suppose that this fails to be true. Then we may fix the smallest $\lambda\in[-\Lambda,\Lambda]$ satisfying $f(\{\lambda\}\times\BS^{d-1})\cap S^N(f(p),\frac 9{20}\epsilon)\ne \emptyset$ and consider the compact set
$$V=[-\Lambda,\lambda]\times\BS^{d-1}.$$
From \eqref{eq-abc} we get that the image under $f$ of each one of the boundary components of $V$ has diameter at most $\frac 1{40}\epsilon$.

Let now $\phi:B^N(f(p),10\epsilon)\to\BB(0,10\epsilon)$ be a $2$-bilischitz chart and consider the set $(\phi\circ f)(V)$. Since we might assume that $f$ is not constant, we get that it is open by Reshetnyak's theorem recalled in (QR1). It follows that 
\[
\D ((\phi\circ f)(V))\subset (\phi\circ f)(\D V)
\]
and hence each component $Z$ of $\D((\phi\circ f)(V))$ has diameter
\[
\diam Z\le\sum_{Y\subset\D V}\diam((\phi\circ f)(Y))\le \frac 4{40}\epsilon.
\]
Since this is true for every boundary component of $(\phi\circ f)(V)$, and since $V$ is compact, it follows that $(\phi\circ f)(V)$ itself is contained in a ball of radius $\frac 4{40}\epsilon$ and thus has at most diameter $\frac 8{40}\epsilon$. Using again that $\phi$ is 2-bilipschitz, we get that
\[
\diam(f(V))\le \frac{16}{40}\epsilon.
\]
Since $p\in V$ this implies that $f(V)\subset B^N(f(p),\frac 8{20}\epsilon)$, contradicting the choice of $\lambda$. Having that $f([-\Lambda,\Lambda])\subset B^N(f(p),\frac 9{20}\epsilon)$, we also have that $\diam(f([-\Lambda,\Lambda]))\le\frac 9{10}\epsilon$, and we are done.
\end{proof}

\subsection*{Quantization of energy}

We fear that the reader might think, as we actually do ourselves, that the word {\em quantization} sounds rather presumptuous. This word, however, reflects well one of the features of quasiregular maps that will play a central role in this paper: {\em under suitable conditions, the energy of quasiregular maps can take only a discrete set of values}. This is what we discuss now. 

First consider the case that domain and target are closed manifolds. By the change of variables \cite[Proposition I.4.1.4(c)]{Rickman-book} and the facts that quasiregular maps are both sense-preserving \cite[Theorem I.4.5]{Rickman-book} and local homeomorphisms almost everywhere \cite[Proposition I.4.1.4(b)]{Rickman-book}, we have that: 

\begin{itemize}
\item[(QR6)] {\em If $M$ and $N$ are closed and oriented Riemannian $d$-manifolds, then 
$$\Vert\mu_f\Vert=\int_M \dmu_f=\int_N\vert f^{-1}(z)\vert d\vol_N(z)=\deg(f)\cdot\vol_N(N)$$
for every quasiregular map $f:M\to N$.}
\end{itemize}


In this paper we will indeed be mostly interested in quasiregular maps between closed manifolds, but maps from open subsets of $\BR^d$ to closed manifolds will also arise naturally. If the domain is $\BR^d$ and the map has finite energy then we are basically in the same situation as in the case that the domain is a closed manifold. This is due to the following fact, which follows easily from \cite[Corollary III.2.10]{Rickman-book} together with the fact that sets of positive measure have positive capacity:

\begin{itemize}
\item[(QR7)] {\em Suppose that $N$ is a closed and oriented Riemannian $d$-manifold. If $f:\BR^d\to N$ is a $K$-quasiregular map with finite total energy $\Vert\mu_f\Vert<\infty$, then $f$ extends to a $K$-quasiregular map $\hat f:\BS^d\to N$. Here we are seeing the sphere $\BS^d$, endowed with the standard round metric, as the 1-point compactification of $\BR^n$.}
\end{itemize}

As a consequence of (QR6) and (QR7), we get that finite energy is also quantized when the domain is $\BR^d$:

\begin{itemize}
\item[(QR8)] {\em Let $N$ be a closed and oriented Riemannian $d$-manifold and let $f:\BR^d\to N$ be a $K$-quasiregular map with finite total energy $\Vert\mu_f\Vert<\infty$. Then 
$$\Vert\mu_f\Vert=D\cdot\vol(N)$$
for some $D\in\BN$.}
\end{itemize}

This last statement obviously fails if we replace $\BR^d$ by an open subset thereof. Our next aim is to bound this failure.

\begin{lem}\label{lem-paininthebutt}
Let $N$ be a closed Riemannian $d$-manifold and $K\ge 1$. For each $\epsilon>0$ and $\omega>0$, there is $\Lambda>0$ such that the following holds: If 
\[
f:\BB(0,\Lambda)\to N
\]
is a $K$-quasiregular map with $\mu_f(\BB(0,\Lambda)\setminus\BB(0,1))<\omega$, then 
$$\big\vert\Vert\mu_f\Vert-D\cdot\vol(N)\big\vert\le\omega+\epsilon$$
for some $D\in\BN$.
\end{lem}
\begin{proof}
Note that we may assume, without loss of generality, that $\epsilon$ is smaller than the injectivity radius of $N$. In fact, so small that for all $x\in N$ the exponential map $\exp_x:T_xN\to N$ is $2$-bilischitz on the set of tangent vectors of length less than $\epsilon$. 

For $k\in \BN$ large, suppose that $\Lambda>2^{k+1}$. Then the restriction of $f$ to one of the standard cylinders $\BB(0,2^{i+1})\setminus\BB(0,2^i)$ has energy less than $\frac\omega k$. Since each one of these cylinders is conformally equivalent to $\BB(0,2)\setminus\BB(0,1)$, it follows from Lemma \ref{lem-neck} that, as long as $\omega/\Lambda$ is small enough, there is $r\in(1,\Lambda)$ such that 
\begin{equation}\label{eq-listening to slade}
\diam_N(f(\BS(0,r)))<\epsilon\ \text{ and }\ \int_{\BS(0,r)}\Vert df\Vert^{d-1}d\vol_{\BS(0,r)}<4^{-d}\cdot\epsilon,
\end{equation}
where the factor $4^{-d}$ has been chosen to offset later the fact that the exponential maps are not isometric but just 2-bilipschitz on the domain we care about. 
 
Now, the first inequality in \eqref{eq-listening to slade} implies that $f(\BS(0,r))\subset B^N(x,\epsilon)$ for some $x\in N$. Consider then the map $\phi:\BB(0,r)\to N$ with 
$$\phi(t\cdot v)=\exp_x(t\cdot\exp_x^{-1}(f(v)))\text{ for }v\in\BS(0,r)\text{ and }t\in[0,1].$$
This map agrees with $f$ on $\BS(0,r)$. We may thus consider the map
\[
F:\BB(0,r)\cup_{\BS(0,r)=\BS(0,r)}\BB(0,r)\to N
\]
which agrees with $f$ on the first ball and with $\phi$ on the second one. Being a map between closed manifolds, $F$ has integral degree $D$. It follows that
$$\int_{\BB(0,r)\cup_{\BS(0,r)=\BS(0,r)}\BB(0,r)}\det(dF) d\vol=D\cdot\vol(N),$$
which means that
$$\mu_f(\BB(0,r))+\int_{\BB(0,r)}\det(d\phi)d\vol=D\cdot\vol(N).$$
Now, taking into account that in the scale we are considering the exponential map is $2$-bilipschitz, we get from the second inequality in \eqref{eq-listening to slade} that
\[
\int_{\BB(0,r)}\det(d\phi)d\vol\le 4^{d-1} \int_{\BS(0,r)}\norm{df}^{d-1} d\vol_{\BS(0,r)}\le \epsilon.
\]
This implies then that $\vert\mu_f(\BB(0,r))-D\cdot\vol(N)\vert<\epsilon$, which thus yields
\begin{align*}
\big\vert\Vert\mu_f\Vert-D\cdot\vol(N)\big\vert
&\le \big\vert\mu_f(\BB(0,r))-D\cdot\vol(N)\big\vert+\mu_f(\BB(0,\Lambda)\setminus\BB(0,r))\\
&\le \epsilon+\omega
\end{align*}
as we wanted to prove.
\end{proof}

\section{Compactness off the poles}
\label{sec:poles}

In \cite{harmonic} we established quasinormality of the family of $K$-quasiregular self-maps of a $d$-sphere of fixed degree, that is, any sequence of such quasiregular maps $\BS^d \to \BS^d$ has a subsequence which converges locally uniformly to a quasiregular map $\BS^d \to \BS^d$ in the complement of a finite set of points, to which we refer as {\em poles}. In this section we revisit this result allowing for general Riemannian manifolds and giving a more precise statement. We begin by giving a precise meaning to the word pole.

\begin{defi*}[Pole]
Let $(f_n:M\to N)$ be a sequence of maps between $d$-dimensional manifolds. A point $p\in M$ is a {\em pole for the sequence $(f_n)$} if there is an open neighborhood $U$ of $p$ in $M$ such that the following hold:
\begin{enumerate}
\item The restriction of $f_n$ to $U\setminus\{p\}$ converges locally uniformly to a map $f:U\setminus\{p\}\to N$.
\item The map $f:U\setminus\{p\}\to N$ extends to a continuous map $U \to N$. 
\item The maps $f_n$ do not converge locally uniformly to $f$ in $U$.
\end{enumerate}
\end{defi*}

In what follows, we use two quantified versions of Proposition 1.2 in \cite{harmonic}, one global and one local. We first state the global one.

\begin{prop}\label{meat2}
Let $M$ and $N$ be closed and oriented Riemannian $d$-manifolds, and fix $K\ge 1$ and $D\in\BN$. Let also 
\[
(f_n:M\to N)
\]
be a sequence of $K$-quasiregular maps of degree $D$. Then, up to passing to a subsequence, there is a finite set $P\subset M$ with the following properties:
\begin{enumerate}
\item The measures $\mu_{f_n}$ converge in the weak-$\ast$ topology to a measure $\mu$ with $\mu(M)=D\cdot\vol(N)$.
\item There is a $K$-quasiregular map $f:M\to N$ with $\mu_f=\mu$ on $M\setminus P$. 
\item The sequence $(f_n)$ converges to $f$ locally uniformly in $M\setminus P$.
\item For each $p\in P$, we have $\mu(p)=d_p\cdot\vol(N)$ for some integer $d_p\ge 1$.
\end{enumerate}
\end{prop}

\begin{bem}
Note that (2), (3) and (4) imply that $P$ is the set of poles of the sequence $(f_n)$. Note also that (4) states that a definite amount of energy is concentrated at each pole.
\end{bem}

\begin{proof}
The first claim follows just from the weak-$\ast$ compactness properties of the space of measures on $M$ and we may move directly to the proofs of the next two claims.

For each $\ell\in \BN$, let $\CB_\ell$ be a finite covering of $M$ by metric balls of radius $2^{-\ell}$ satisfying $\frac 12B\cap\frac 12B'=\emptyset$ for all distinct $B,B'\in\CB_\ell$. Note that since the balls $\frac 12B$ with $B\in \CB_{\ell}$ are pair-wise disjoint, a point in $M$ belongs to at most $4^d$ balls in $\CB_{\ell}$ as long as $\ell$ is over some threshole. 

Now, for $n\in \BN$, let $\CB_{\ell,n}^b\subset \CB_\ell$ be the subcollection consisting of balls $B\in \CB_\ell$ with $\mu_{f_n}(B)>\frac 12\vol(N)$. By passing to a subsequence of $(\CB_{\ell,n}^b)_n$, we may assume that there is a fixed subset $\CB^b_\ell\subset\CB_\ell$ with $\CB_{\ell,n}^b=\CB_{\ell}^b$ for all $n\in \BN$. The total bound on the energy $\Vert\mu_{f_n}\Vert=D\cdot\vol(N)$, the lower bound on the energy on the balls in $\CB_{\ell}^b$, together with the fact that each point belongs to at most $4^d$ balls, implies the upper bound
\[
\vert\CB^b_\ell\vert\le 2\cdot 4^d\cdot D
\]
on the cardinality of $\CB^b_\ell$ for all $\ell\in \BN$ large enough.

This uniform cardinality bound implies also that, up to passing to a further subsequence, we may assume that, when $\ell\to\infty$, the sets $\bigcup \CB_{\ell}^b$ converge in the Hausdorff topology to a finite set $P \subset M$. 

Let now $U\subset M$ be an open set whose closure $\overline{U}$ does not meet the set $P$. For each $\ell\in \BN$, let $\CB_\ell(U)$ be the collection of balls in $\CB_\ell$ which meet $U$. Since $P$ and $U$ have positive distance, there exists an index $\ell_0\in \BN$ with the property that the union of the balls in $\CB_\ell(U)$ is disjoint of the balls in $\CB_\ell^b$ for all $\ell\ge\ell_0$. Now, for each $B\in \CB_{\ell_0}(U)$, we have that $\mu_{f_n}(B)<\frac 12\vol(N)$ for all $n$ large enough. It thus follows from (QR5) above that the sequence $(f_m|_B)$ is equicontinuous for each $B\in \CB_\ell(U)$. In particular, it has a subsequence which converges uniformly to some $f_B \colon B \to N$. The limiting map $f_B$ is $K$-quasiregular by the first part of (QR3).
Now, a standard diagonal argument yields a subsequence $(f_{m_j})$ tending locally uniformly on $M\setminus P$ to a $K$-quasiregular map $f \colon M\setminus P \to N$. Since $f$ has finite energy, it extends to a $K$-quasiregular map $M \to N$ by removability of isolated singularities (\cite[Theorem III.2.8]{Rickman-book}). The uniform convergence on compacta of $(f_n)$ to $f$ in $M\setminus P$ implies that the measures $\mu_{f_n}$ converge to $\mu_f$ on $M\setminus P$. Thus we have proved both (2) and (3).

To prove (4), let $p\in P$ and let $\epsilon>0$ be small but otherwise arbitrary. Choose also $\delta>0$ with 
\[
\mu_f(B^M(p,\delta))<\epsilon,
\]
where $f:M\to N$
is the limiting map provided by (2). Also, assume that the exponential map 
$$\exp_p:T_pM\to M$$
is $2$-bilipschitz on the set of vectors of length at most $\delta$.

Suppose now that $\Lambda>0$ is very large and note that there exists $n_0\in \BN$ satisfying
\[
\mu_{f_n}\left(B^M(p,\delta)\setminus B^M\left(p,\frac 1\Lambda\delta\right)\right)<\epsilon
\]
for all $n\ge n_0$. Identifying $T_pM$ and $\BR^d$ consider now the map
\[
F_n:\BB(0,\Lambda)\to N,\ \ v\mapsto f_n\left(\exp_p\left(\frac 1\Lambda\delta v\right)\right).
\]
These maps are $2^d\cdot K$-quasiregular, and satisfy 
$$\Vert\mu_{F_n}\Vert=\mu_{f_n}(B^M(p,\delta))\ \text{ and }\ \mu_{F_n}\left(\BB(0,\Lambda)\setminus\BB(0,1)\right)\le\epsilon$$
for all $n\ge n_0$. It follows from Lemma \ref{lem-paininthebutt} that, as long as $\Lambda$ was chosen large enough, there exists, for each $n\in \BN$, $D_n\in\BN$ with 
$$\big\vert\Vert\mu_{F_n}\Vert-D_n\cdot\vol(N)\big\vert\le 2\epsilon$$
for all $n\ge n_0$. It follows that there is some $D\in\BN$ with 
$$\vert\mu(B^M(p,\delta))-D\cdot\vol(N)\vert\le 2\epsilon.$$
Since $\epsilon$ was arbitrary, and since $\delta$ could be replaced by any smaller number, we have
that $\mu(\{p\})=D\cdot\vol(N)$, as we needed to prove.
\end{proof}

We will use a local version of Proposition \ref{meat2}, which applies when the maps $f_n$ in question are defined on larger and larger balls in $\BR^d$. This poses a number of annoying problems. First, we cannot speak about degree. To solve this issue, we switch this topological assumption to an assumption on the total energy of the involved maps. Also, since the maps are defined on non-compact spaces, we have to deal with the lack of compactness of the space of measures. 

Condition (a) in the next proposition deals with this problem, already making the statement more cumbersome. But there is more. We also want to include some condition ensuring that the limit is not too degenerate. This is the reason for condition (b).
After all these warnings, the reader might be ready to confront the mightly unattractive Proposition \ref{meat}.

\begin{prop}\label{meat}
Let $N$ be a closed and oriented Riemannian $d$-manifold, $(\Lambda_n)$ a sequence of positive numbers tending to $\infty$, $K\ge 1$, $C\ge 0$, and $0<\omega<\frac 13\vol(N)$. For each $n\in \BN$, let $f_n:\BB(0,\Lambda_n)\to N$ be a $K$-quasiregular map satisfying
\begin{itemize}
\item[(a)] $\mu_{f_n}(\BB(0,\Lambda_n)\setminus\BB(0,1))\le 2\omega$, and
\item[(b)] $\mu_{f_n}(\BB(z,\frac 14))\le\norm{\mu_{f_n}} - \omega$
for all $z\in\BB(0,\Lambda_n)$.
\end{itemize}
Suppose also that  
\[
\mu_{f_n}(\BB(0,\Lambda_n)) \to C
\]
as $n\to \infty$.

Then, up to passing to a subsequence, there is a finite set $P\subset\overline{\BB(0,1)}\subset\BR^d$ with the following properties:
\begin{enumerate}
\item The measures $\mu_{f_n}$ converge in the weak-$\ast$ topology to a measure $\mu$ satisfying
$$C-2\omega\le\mu(\BR^d)\le C.$$
\item There is a $K$-quasiregular map $f:\BR^d\to N$ with $\mu_f=\mu$ on $\BR^d\setminus P$. 
\item The maps $(f_n)$ converge to $f$ locally uniformly in $\BR^d\setminus P$.
\item For each $p\in P$, we have $\mu(p)=d_p\cdot\vol(N)$ for some integer $d_p\ge 1$.
\item If $f$ is constant, then $\vert P\vert\ge 2$.
\end{enumerate}
\end{prop}

The proofs of (1)--(4) in Proposition \ref{meat} are basically identical to the proofs of the corresponding statements in Proposition \ref{meat2} above -- we leave the details to the reader. Claim (5) follows directly from (4) because the maps satisfy condition (b). Indeed, given $p\in P$, we have, by (b) and the weak-$\ast$ convergence, that $\omega \le \mu(\BR^d\setminus \{p\})$. Since $f$ is constant, the measure $\mu_f$ is trivial. Hence $P\ne \{p\}$ by (4). This concludes the discussion of Proposition \ref{meat}.\qed

\subsection*{Renormalisation near a pole}

Conditions (a) and (b) in Proposition \ref{meat} are in some way unnatural, meaning that they will only be satisfied after we modify the sequence in a maybe arbitrary way in an actually occurring situation. In fact, these two conditions are going to arise after we suitably rescale the domain near a pole. The following is the precise statement we will use.

\begin{prop}\label{spread}
Let $N$ be an oriented Riemannian $d$-manifold, $K\ge 1$, $0<\epsilon<\omega<\frac1{10}\vol(N)$, and suppose that also $\epsilon<\frac 18$. Then, for each $K$-quasiregular map $f:\BB(0,1) \to N$ of finite energy satisfying $\Vert\mu_f\Vert\ge\vol(N)$ and 
\begin{equation}
\label{mi gato se llama chichu}
\mu_f(\BB(0,\epsilon))\ge \Vert\mu_f\Vert-\epsilon,
\end{equation}
there exists $r\ge\frac 1{3\epsilon}-1$ and an affine conformal map $\phi:\BB(0,r)\to\BB(0,1)$ for which the $K$-quasiregular map $f\circ\phi:\BB(0,r)\to N$ has the following properties:
\begin{enumerate}
\item $\mu_{f\circ\phi}(\BB(0,r)\setminus\BB(0,1))\le 2\omega$,
\item $\mu_{f\circ\phi}(\BB(z,\frac 14))\le\Vert\mu_f\Vert-\omega$ for all $z\in\BB(0,r)$, and
\item $\Vert \mu_{f\circ\phi}\Vert\ge\Vert\mu_f\Vert-\epsilon$.
\end{enumerate}
\end{prop}

Before launching the proof of Proposition \ref{spread}, we recall briefly what is the {\em centre of mass} of a non-atomic measure $\mu$ on $\BR^d$ with finite first moment $\int_{\BR^d}\vert x\vert d\mu(x)<\infty$ and whose support is not contained in a line. Consider the function
\[
\Phi_\mu:\BR^d\to\BR,\ x\mapsto\int_{\BR^d}\vert x-z\vert d\mu(z).
\]
The properties of $\mu$ translate to properties of $\Phi_\mu$ as follows. The finiteness of the first moment of $\mu$ implies that the function $\Phi_\mu$ is well-defined. Since $\mu$ has no atoms, $\Phi_\mu$ is smooth. Since the support of $\mu$ is not contained in a line, $\Phi_\mu$ is strictly convex. Finally, since $\Phi_\mu$ is proper, it has a unique minimum $x_0$, the {\em centre of mass} of $\mu$. We may characterize $x_0$ also as the unique point where the gradient of $\Phi_\mu$ vanishes, that is, $x_0$ is the centre of mass if and only if
\begin{equation}
\label{eq:center-2}
\int_{\BR^d}\frac {x_0-z}{\vert x_0-z\vert}d\mu(z)=0.
\end{equation}
The centre of mass satisfies what one could call the gravity principle: {\em if most of the measure is in northern Sweden, then the centre of mass is in a Nordic country}. 

\begin{lem}\label{vittula}
Suppose that $\mu$ is a non-atomic measure on $\BR^d$ with finite first moment and whose support is not contained in a line. If $\mu(\BB(0,1))>\frac 34\mu(\BR^d)$, then the centre of mass of $\mu$ is contained in $\BB(0,2)$.
\end{lem}
\begin{proof}
It suffices to show that $x_0\in\BR^d$ with $\vert x_0\vert\ge 2$ is not the centre of mass of $\mu$. In fact, for any such $x_0$ we have 
$$\left\langle x_0,x_0-z\right\rangle\ge\vert x_0\vert\cdot(\vert x_0\vert-1)\ \text{ and }\ \vert x_0-z\vert\le\vert x_0\vert +1$$
for every $z\in\BB(0,1)$. Thus 
\begin{equation}
\label{eq:angle}
\begin{split}
\left\langle x_0,\int_{\BR^d}\frac{x_0-z}{\vert x_0-z\vert}d\mu(z)\right\rangle&\ge \frac {\vert x_0\vert\cdot(\vert x_0\vert -1)}{\vert x_0\vert+1}\cdot\mu(\BB(0,1))-\vert x_0\vert\cdot\mu(\BR^d\setminus\BB(0,1))\\
&\ge\vert x_0\vert\cdot\left(\frac 13\cdot\mu(\BB(0,1))-\mu(\BR^d\setminus\BB(0,1))\right).
\end{split}
\end{equation}
The condition $\mu(\BB(0,1))>\frac 34\mu(\BR^d)$ implies that $\mu(\BR^d\setminus\BB(0,1))\le \frac 14\mu(\BR^d)$ and thus the quantity on the right in \eqref{eq:angle} is positive. In particular, this shows that 
\[
\int_{\BR^d}\frac {x_0-z}{\vert x_0-z\vert}d\mu(z)\neq 0,
\]
which means that $x_0$ is not be the centre of mass of $\mu$.
\end{proof}

We are ready to prove Proposition \ref{spread}. 

\begin{proof}[Proof of Proposition \ref{spread}]
Let $x_0\in \BR^d$ be the center of mass of $\mu_f$, and note that $\vert x_0\vert<2\epsilon$ by Lemma \ref{vittula}. Given $\lambda\ge 1$ consider the affine scaling map
\[
T^\lambda_{x_0}: \BR^d \to \BR^d,\ \ x\mapsto \lambda (x-x_0) + x_0,
\]
and the function
\[
\Psi_f: [1,\infty) \to [0,\norm{\mu_f}),\ \ \lambda\mapsto \left((T^\lambda_{x_0})_*\mu_f\right)(T^\lambda_{x_0}(\BB(0,1))\setminus\BB(0,1)).
\]
Since $\mu_f$ has no atoms, $\Psi_f$ is continuous. Since $\Psi_f(1) = 0$ and $\Psi_f(\lambda) \to \norm{\mu_f}$ as $\lambda \to \infty$,  there exists $\lambda_f\ge 1$ for which
\begin{equation}
\label{eq:match}
\Psi_f(\lambda_f) = \left((T^{\lambda_f}_{x_0})_*\mu_f\right)\left(T^{\lambda_f}_{x_0}(\BB(0,1))\setminus\BB(0,1)\right) = 2\omega.
\end{equation}
To satisfy \eqref{eq:match}, the image of $\BB(0,\epsilon)$ under $T^{\lambda_f}_{x_0}$ needs to meet the complement of $\BB(0,1)$. Thus $\lambda_f \ge \frac{1}{3\epsilon}$.


Having fixed $\lambda_f$, we define
\[
\phi = (T^{\lambda_f}_{x_0})^{-1}:\BB(0,r)\to \BB(0,1),\ \ x \mapsto \frac{1}\lambda_f(x-x_0)+x_0,
\]
where 
\[
r=\lambda_f(1-2\epsilon)\ge\frac 1{3\epsilon}-1.
\]
Note also, keeping in mind that $\vert x_0\vert<2\epsilon$, that 
\begin{equation}\label{eq-inclusion1}
\BB(0,r)=\BB(0,\lambda_f(1-2\epsilon)) \subset \BB((1-\lambda_f)x_0, \lambda_f) = T^{\lambda_f}_{x_0}\BB(0,1)
\end{equation}
and
\begin{equation}\label{eq-inclusion2}
T^{\lambda_f}_{x_0}\BB(0,1/2) = \BB((1-\lambda_f)x_0, \lambda_f/2) \subset 
\BB(0,r).
\end{equation}

We are now ready to show that this map $\phi$ has the required properties. First we observe that, since $\phi$ is conformal, $f\circ \phi: \BB(0,r)\to N$ is $K$-quasiregular and satisfies $\mu_{f\circ \phi} = (T^{\lambda_f}_{x_0})_*\mu_f$ by (QR2). By \eqref{eq-inclusion1}, we also have 
\begin{align*}
\mu_{f\circ \phi}(\BB(0,r)\setminus \BB(0,1)) & = (T^{\lambda_f}_{x_0})_*\mu_f(\BB(0,r)\setminus \BB(0,1)) \\
&\le (T^{\lambda_f}_{x_0})_* \mu_f(T^{\lambda_f}_{x_0}(\BB(0,1))\setminus \BB(0,1)) \\
&=\Psi_f(\lambda_f)= 2\omega.
\end{align*}
Thus (1) holds.

By-passing (2) for a moment, let us prove (3). To do so observe that from \eqref{eq-inclusion2} we get
\begin{align*}
\norm{\mu_{f\circ \phi}} &= (T^{\lambda_f}_{x_0})_*\mu_f(\BB(0,r)) = \mu_f((T^{\lambda_f}_{x_0})^{-1}(\BB(0,r)) \\
&\ge \mu_f(\BB(0,1/2)) \ge \mu_f(\BB(0,\epsilon)) \ge\norm{\mu_f}-\epsilon.
\end{align*}
Condition (3) follows. 

We prove (2) by contradiction. In the light of (3), this means that there is $z$ with 
$$\mu_{f\circ\phi}(\BB(z,\frac 14))\ge\Vert\mu_{f\circ\phi}\Vert-\omega\ge \norm{\mu_f}-\epsilon-\omega>\Vert\mu_f\Vert-2\omega.$$
It follows from (1) that $\BB(z,\frac 14)$ has to meet $\BB(0,1)$. But since $\mu_f(\BB(z,\frac 14))\ge\mu_{f\circ\phi}(\BB(z,\frac 14))\ge\frac 34\Vert\mu_f\Vert$, we get from Lemma \ref{vittula} that the centre of mass $x_0$ of $\mu$ belongs to the ring $\{z\in\BR^d \colon  \frac {3}2\ge\vert x\vert\ge\frac 12\}$. This is a contradiction because $\vert x_0\vert< 2\epsilon<\frac 14$. We have proved (2), and thus Proposition \ref{spread}.
\end{proof}

\section{Nodal manifolds}\label{sec-nodal}

Given two topological spaces $X$ and $Y$ and two points $x\in X$ and $y\in Y$, let $X*_{x=y}Y$ be the topological space obtained by identifying the points $x$ and $y$. Now, given $d\ge 2$, let $\CN^d$ be the smallest collection of topological spaces containing all $d$-dimensional manifolds and which is closed under the following operation: 
\begin{itemize}
\item [] For $X,Y\in\CN^d$ and $x\in X$ and $y\in Y$ points, $X*_{x=y}Y\in\CN^d$. 
\end{itemize}
We refer to the elements of $\CN^d$ as {\em topological nodal $d$-manifolds}. 

Let $X$ be a topological nodel $d$-manifold. A point $x\in X$ is {\em regular} if $X\setminus \{x\}$ is connected and it is {\em singular} otherwise; equivalently, $x\in X$ is a regular iff it is a manifold point of $X$. A nodal manifold $X$ is {\em pure} if $X\setminus \{x\}$ has two connected components for every singular point $x$. Proposition \ref{prop-characterise nodal mf}, whose proof we leave to the reader, gives an intrinsic characterisation of pure nodal manifolds.

\begin{prop}\label{prop-characterise nodal mf}
Let $d\ge 2$. A connected, Hausdorff, and paracompact topological space $Z$ is a pure topological nodal $d$-manifold if and only if it contains a finite set $P\subset Z$ satisfying the following properties:
\begin{itemize}
\item Each point $p\in P$ has a neighborhood $U\subset Z$ such that the pair $(U,p)$ is homeomorphic to $(\BR^d\times\{0\}\cup\{0\}\times\BR^d,(0,0))\subset(\BR^d\times\BR^d,(0,0))$.
\item $Z\setminus P$ is an $d$-dimensional manifold with precisely $\vert P\vert+1$ connected components.\qed
\end{itemize}
\end{prop}

The set $\Sing(X)$ of singular points of a nodal manifold $X$ is finite and its complement, the {\em regular part} of $X$, is an honest $d$-manifold. In fact, if $V\in\pi_0(X\setminus\Sing(X))$ is a connected component of the regular part of $X$, then the closure $\overline{V}\subset X$ of $V$ in $X$ is also a $d$-dimensional manifold. We refer to the closures of the connected components of the regular parts as the {\em strata} of $X$ and denote by $\Strata(X)$ the collection of all strata of $X$.

A nodal manifold is {\em closed, orientable, oriented} or {\em smooth} if all its strata are. Similarly, a Riemannian metric on a smooth nodal manifold is nothing other than a Riemannian metric on the disjoint union of the strata. A continuous map $f:X\to N$ is {\em smooth} if it is smooth on each stratum of $X$ -- similarly for other kinds of maps such as H\"older, Lipschitz, quasiregular, etc. Other notions extend in the same way. For example, the energy density $\mu_f$ of a quasiregular map $f:X\to N$ from a Riemannian nodal manifold $X$ to a Riemannian manifold $N$ is the measure whose restriction to each stratum $V$ is the energy density of $f\vert_V$.

\begin{defi*}[Pinching map]
Suppose that $X$ is a pure smooth oriented nodal $d$-manifold and that $M$ is a smooth oriented $d$-manifold. A surjective continuous proper map $\pi: M\to X$ is a {\em pinching map} if the following two conditions are satisfied:
\begin{itemize}
\item $\pi^{-1}(x)$ is a smoothly embedded sphere of dimension $d-1$ for each $x\in\Sing(X)$, and
\item if $V\in \pi_0(X\setminus\Sing(X))$ is a component of the regular part then the restriction of $\pi$ to $\pi^{-1}(V)$ is an orientation preserving diffeomorphism onto $V$.
\end{itemize}
\end{defi*}

Note that, if $\pi:M\to X$ is a pinching map, then $M$ is homeomorphic to the connected sum of the strata of $X$. 

In some sense, the point of this paper is to explain how do nodal manifolds arise naturally when studying degenerating sequences of quasiregular maps. This heurstic principle is encoded into the following definition, which we state in the topological form before switching to the quasiconformal category.

\begin{defi*}[Convergence]
Let $M$ and $N$ be $d$-manifolds and let $X$ be a pure nodal $d$-manifold. A sequence 
\[
(f_n:M\to N)
\]
of continuous maps {\em converges} to a continuous map $f:X\to N$ if there are pairwise properly homotopic, uniformly proper pinching maps $\pi_n:M\to X$ for which the maps 
\[
f_n\circ\pi_n^{-1}:X\setminus\Sing(X)\to N
\]
converge uniformly on compacta to $f$. 

The convergence of $(f_n)$ to $f$ is {\em tight} if for every $\delta>0$ there exist $n_0\in \BN$ and an open neighborhood $A\subset X$ of $\Sing(X)$ such that, for every $n\ge n_0$ and $U \in \pi_0(A)$, holds $\diam(f_n(\pi_n^{-1}(U)))<\delta$.
\end{defi*}

Here, the assumption that the sequence $\pi_n:M\to\ X$ is {\em uniformly proper} just means that for every $C_1\subset M$ compact there are again compact sets $C\subset X$ and $C_2\subset M$ with
$$C_1\subset\pi_n^{-1}(C)\subset C_2$$
for every $n$. This condition is obviously only relevant if $M$ is non-compact. In this paper we will only find $M$ non-compact if $M=\BR^d$. 

\begin{bem}
If we want to stress the role of the pinching maps $(\pi_n)$ in the converge, we say that $(f_n)$ \emph{converges to $f$ along the sequence $(\pi_n)$}.
\end{bem}

We stress that we are assuming that the pinching maps $\pi_n:M\to X$ in this definition of convergence are pairwise homotopic. After a moment of reflection it becomes clear that some condition of the maps is needed. But then, once one is convinced that this is the case, it is quite possible that one finds the condition of being pairwise homotopic too feeble. It is however not clear what stronger condition one could impose. For example, if one were to assume that the maps $\pi_n$ are equicontinuous then Theorem \ref{main} would definitively fail to be true. But yes, just imposing that the maps $\pi_n$ are pairwise homotopic is a rather weak assumption. For instance, it is by itself not even sufficient to ensure that the maps $f_n$ belong to finitely many homotopy classes. In fact, we encourage the reader to produce an example of a sequence of maps $f_n:M\to N$ converging to a map $f:X\to N$ while having that $f_n$ and $f_m$ represent different homotopy classes for all $n\neq m$. Such examples can however only arise if the convergence fails to be tight:

\begin{lem}\label{lem-homotopy}
Let $M$ and $N$ be closed $d$-manifolds, $X$ a nodal manifold, and $D\ge 1$. If a sequence $(f_n:M\to N)$ of continuous maps of degree $D$ converges tightly to a map $f:X\to N$, then the maps $f_n$ represent only finitely many homotopy classes and $D=\sum_{V}\deg(f\vert_V)$, where the sum is over the strata of $X$.
\end{lem}

\begin{proof}
Let $\pi_n:M\to X$ be the pinching maps in the definition of convergence and choose $\epsilon>0$ smaller than one tenth of the injectivity radius of $M$. Since the convergence is tight, there is a neighborhood $A$ of $\Sing(X)$ in $X$ with $\diam(f_n(\pi_n^{-1}(U)))<\epsilon$ for each $U\in \pi_0(A)$ and for all $n\in \BN$ large enough. Thus, by passing to a smaller neighborhood $A$ of $\Sing(X)$ if necessary, we may assume that $\diam(f(U)) < \epsilon$ for all $U\in \pi_0(A)$. 

The bound on the diameter of the image of the components of $A$ implies that, for all $n\in \BN$ large enough, the map $f_n$ is homotopic, via a homotopy whose tracks have length less than $\epsilon$, to another map $f_n':M\to N$ which is constant on each connected component of $\pi_n^{-1}(\Sing(X))$. The maps $f'_n$ descend to $X$ in the sense that there are maps $\hat f_n:X\to N$ satisfying $f_n'=\hat f_n\circ\pi_n$. 

Convergence of the maps $f_n$ to $f$ implies that for all $n$ large enough, we have $d_N(f_n(\pi_n^{-1}(x)),f(x))\le\epsilon$ for all $x\in X\setminus A$. The triangle inequality then implies that $d_N(\hat f_n(x),f(x))\le 2\epsilon$ for all $x\in X\setminus A$. Since each connected component of $\hat f_n(A)$ and $f(A)$ have diameter at most $\epsilon$ for large $n$, we get from the triangle inequality that $d_N(\hat f_n(x),f(x))<\inj(N)$ for all $x\in X$ and for all $n\in \BN$ large enough. This implies that $\hat f_n$ and $f$ are homotopic for $n$ large. Since $f_n$ is homotopic to $f_n'=f_n\circ\pi_n$, we have thus proved that $f_n$ and $f\circ\pi_n$ are homotopic for all $n\in \BN$ large enough. Since the maps $\pi_n$ are all pairwise homotopic, it follows that the maps $f_n$ and $f_m$ are pairwise homotopic to each other if $n$ and $m$ are large enough. The first claim follows. 

To prove the second claim, suppose for a moment that $f_n$ and $f$ are smooth. Since the map $\pi_n$ is an orientation preserving diffeomorphism outside of $\pi_n^{-1}(X)$ we get, for generic $y\in N$, that 
\begin{align*}
\deg(f_n)
&=\deg(f\circ\pi_n)=\sum_{x\in(f\circ\pi_n)^{-1}(y)}\mathrm{sign}(\det d(f\circ\pi_n)_x)\\
&=\sum_{x\in(f\circ\pi_n)^{-1}(y)}\mathrm{sign}(\det df_{\pi_n(x)})\cdot\mathrm{sign}(\det d(\pi_n)_x)\\
&=\sum_{z\in f^{-1}(y)}\mathrm{sign}(\det df_z)=\sum_{V\in \Strata(X)}\deg(f\vert_V).
\end{align*}
This proves the claim if $f_n$ and $f$ are smooth. The general case follows by smooth approximation of the maps $f_n$ and $f$ in the strata and the fact that close enough smooth approximations are homotopic to the original maps. 
\end{proof}

\subsection*{Convergence of quasiregular maps}

It is obvious from the definition of convergence of a sequence of maps between manifolds to a map with domain a nodal manifold, that some geometric condition for the involved pinching maps is needed in order to have the property that the limiting map is $K$-quasiregular if the members of the sequence are. This is the point of the following definition:

\begin{defi*}[Asymptotically conformal pinching maps]
A sequence 
\[
(\pi_n \colon M\to X)
\]
of pinching maps from an oriented Riemannian manifold to a pure oriented Riemannian nodal manifold  is \emph{asymptotically conformal} if there exists a sequence of positive numbers $(K_n)$ tending to $1$ and a sequence $(U_i)$ of relatively compact open sets in $X\setminus\Sing(X)$ with 
$$U_n\subset U_{n+1}\text{ for all }n\text{, with }X\setminus\Sing(X)=\cup U_n,$$ 
and such that $\pi_n^{-1}$ is $K_n$-quasiconformal on $U_n$ for all $n$.
\end{defi*}

After setting this up, we have that limits of $K$-quasiregular maps are $K$-quasiregular. 
 
\begin{lem}
\label{lemma:pinching-K}
Let $M$ and $N$ be oriented Riemannian $d$-manifolds and let $(f_n:M\to N)$ be a sequence of $K$-quasiregular mappings. If $(f_n)$ converges along a sequence $(\pi_n\colon M \to X)$ of asymptotically conformal pinching maps to a map $f \colon X\to N$, then $f$ is $K$-quasiregular and satisfies $\Vert\mu_f\Vert\le\limsup\Vert\mu_{f_n}\Vert$.  
\end{lem}

\begin{proof}
It suffices to show that $f$ is $K$-quasiregular in each stratum $V$. Let $x\in V\setminus \Sing(X)$. Then there exists a relatively compact neighborhood $U$ of $x$ not meeting $\Sing(X)$. Thus, for $n\in \BN$ large enough, $\pi_n$ is $K_n$-quasiconformal in $\pi_n^{-1}(U)$ and $f_n \circ \pi_n^{-1}|_U \colon U \to N$ is $K_n K$-quasiregular. Thus, by (QR3), the limiting map $f$ is $K$-quasiregular on $U$, and hence on $V\setminus \Sing(X)$. Since $V\cap \Sing(X)$ consists of isolated points and $f$ is continous in $V$, it now follows that $f$ is quasiregular in $V$ by the removability of isolated singularities \cite[Theorem III.2.8]{Rickman-book}. The bound for the energy
$$\Vert\mu_f\Vert\le\limsup\Vert\mu_{f_n}\Vert$$
also follows from (QR3).
\end{proof}

Before moving on, note that nodal manifolds $X$ for which there is a sequence of asymptotically conformal pinching maps $M\to X$ are rather special. For example, in the case that $M=\BR^d$ we get from the definition that every connected component $V$ of $X\setminus\Sing(X)$ has an exhaustion by nested open sets $U_n$, each of which can be $K_n$-quasiconformally embedded into $\BR^d$ for some sequence $K_n\to 1$. It follows that $V$ can be conformally embedded into $\BR^d$. And then the removability of singularities \cite[Theorem III.2.8]{Rickman-book} implies that in fact every stratum can itself be conformally embedded into $\BS^d$. Since all bubbles of $X$ are compact, we get that all bubbles are conformally equivalent to $\BS^d$.

\begin{lem}\label{kor-nodalRn-conformal}
Suppose that a pure Riemannian nodal manifold $X$ is such that there is a sequence $(\pi_n:\BR^n\to X)$ of asymptotically conformal pinching maps. Then each bubble is conformal to $\BS^d$.\qed
\end{lem}

\subsection*{Tight convergence from renormalisation at poles}
Recall that by Proposition \ref{meat2} a sequence of $K$-quasiregular maps of fixed degree has a subsequence which converges outside of a finite set of poles. Nodal manifolds arise when we blow up the poles of convergent sequences of maps. We discuss next how to isolate what happens near the poles of what happens elsewhere. To keep the discussion structured we need two definitions:

\begin{defi*}[Blow-up sequence]
A sequence $(\sigma_n \colon \BR^d \to M)$ of smooth embeddings is a {\em blow-up sequence on a Riemannian manifold $M$ at point $p\in M$} if there exist sequences $(r_n)$, $(K_n)$, and $(\Lambda_n)$ tending respectively to $0$, $1$, and $\infty$ as $n\to \infty$ so that
\begin{itemize}
\item $\sigma_n(\BR^d) = B^M(p,r_n)$ and
\item $\sigma_n|_{\BB(0,\Lambda_n)}$ is $K_n$-quasiconformal
\end{itemize}
for each $n\in \BN$.
\end{defi*}

\begin{defi*}[Renormalisation]
Let $p\in M$ be a pole for a sequence $(f_n \colon M\to N)$ of $K$-quasiregular maps. A blow-up sequence $(\sigma_n \colon \BR^d \to M)$ at $p$ is a \emph{renormalisation sequence} if 
\begin{itemize}
\item there exists a nodal Riemannian $d$-manifold $X^p$ with main stratum isometric to $\BR^d$ and a map $F^p \colon X^p \to N$ such that the sequence $(f_n \circ \sigma_n \colon \BR^d \to N)$ converges tightly to a map $F^p$ along a sequence of asymptotically conformal pinching maps $\pi_n:\BR^d\to X^p$, such that
\item for each $\epsilon>0$ and $\delta_0>0$ there are $\delta\in(0,\delta_0)$ and $Z\subset X^p$ compact with
$$\mu_{f_n}\left(B^M(p,\delta)\setminus \pi^{-1}_n(X^p\setminus Z)\right) < \epsilon$$
for each $n\in \BN$. 
\end{itemize}
We say that $F^p: X^p\to N$ is the \emph{$(\sigma_n)$-blow-up of $(f_n)$ at $p$}. 
\end{defi*}


One should think of the second condition in the definition of {\em renormalisation sequence} as asserting that the "convergence of the maps $f_n\circ\sigma_n:\BR^d\to N$ is tight near infinity" -- compare with Proposition \ref{lem-neck2} and with the proof of the following lemma:

\begin{lem}
\label{lemma:blow-up-limit}
Fix $K\ge 1$ and $C>0$, suppose that $N$ is a closed manifold, and let $(f_n \colon M \to N)$ be a sequence of $K$-quasiregular maps with uniformly bounded energy $\Vert\mu_{f_n}\Vert\le C$. Assume that $(f_n)$ converges locally uniformly to $f\colon M \to N$ in the complement of a finite set of poles $P\subset M$. Fix also $p\in P$, let $(\sigma_n)$ be a renormalisation sequence for $(f_n)$ at $p\in M$, and let 
$$F^p \colon X^p \to N$$ 
be the $(\sigma_n)$-blow-up of the sequence $(f_n)$. Then the following holds:
\begin{itemize}
\item The main stratum of $X^p$ is isometric to $\BR^d$ and all the others are conformally equivalent to $\BS^d$. 
\item The map $F^p$ is $K$-quasiregular and extends continuously to a $K$-quasiregular map $\widehat F^p \colon \widehat X^p \to N$ on the one-point compactification $\widehat X^p = X^p\cup \{\widehat p\}$ of $X^p$. 
\item Moreover, $\widehat F^p(\widehat p) = f(p)$. 
\end{itemize}
\end{lem}

\begin{proof}
We get from the definition of renormalisation sequence that the main stratum of $X^p$ is isometric to $\BR^d$ -- all the bubbles are conformally equivalent to $\BS^d$ by Lemma \ref{kor-nodalRn-conformal}. We have proved that first claim.

Lemma \ref{lemma:pinching-K} implies that the limiting map $F^p$ is $K$-quasiregular and has finite energy $\Vert\mu_{F^p}\Vert\le C$. Since the main stratum of $X^p$ is isometric to $\BR^d$, and since $\hat p=\hat X^p\setminus X^p$ is nothing other than the point at infinity of $\BR^d$, we get from (QR7) that the map $F^p$ extends continuously to a $K$-quasiregular map $\widehat F^p \colon \widehat X^p \to N$. We are done with the second claim.

So far we have only used the first condition in the definition of renormalisation sequence. We now use the second one, but before doing so we identify the main stratum of $X^p$ with $\BR^d$ and accordingly $\hat p=\BS^d\setminus\BR^d$. Now, the second condition in the definition of renormalisation sequence implies, together with Proposition \ref{lem-neck2}, that for each $\epsilon>0$ there exists $\delta>0$, $R>0$, and $n_0\in \BN$ with
\[
\diam \left(f_n ( B^M(p,\delta)\setminus \pi_n^{-1}(\BB(0,R))) \right) < \epsilon
\]
for each $n\ge n_0$. Now, for $R$ large, points in $\pi_n^{-1}(\BS(0,R))$ get mapped by $f_n$ near the image $\hat F^p(\hat p)$ of the point at infinity under the extension $\hat F_p$. And, since the maps $f_n$ converge uniformly on compacta in $U\setminus\{p\}$ for some open neighborhood $U$ of $p$ in  $M$ we get that, for $\delta$ small, points in $S^M(p,\delta)$ get mapped by $f_n$ near $f(p)$. It thus follows that $\hat F^p(\hat p)$ is close to $f(p)$. Decreasing $\delta$ and increasing $R$, we improve our upper bound on the distance between $\hat F^p(\hat p)$ and $f(p)$, getting in the limit that both points agree.
\end{proof}

Suppose now that $(f_n\colon M\to N)$ is a sequence of quasiregular maps between Riemannian $d$-manifolds which converges locally uniformly on $M\setminus P$ to a map $f\colon M \to N$, where $P$ is the finite set of poles of $(f_n)$. Suppose also that for each $p\in P$, we have fixed a renormalizing sequence $(\sigma^p_n)$ for $(f_n)$ and let $\widehat F^p:\widehat X^p\to N$ be the associated blow up. Consider the nodal $d$-manifold
\begin{equation}\label{pekkaisapain}
X = (M \sqcup_{p\in P} \widehat X^p)\big/{\sim},
\end{equation}
where the equivalence relation $\sim$ is the minimal one identifying each pole $p\in P$ with the point $\hat p=\hat X^p\setminus X^p$, and consider the map
\begin{equation}\label{noooojustkidding}
F \colon X \to N
\end{equation}
defined by $F|_M = f$ and $F|_{\widehat X^p} = F^p$ for each $p\in P$. The map $F$ is well-defined by Lemma \ref{lemma:blow-up-limit}. We think of $F$ as the \emph{total blow} up of $(f_n)$ associated to the chosen renormalizing sequences.

After all these definitions, the following lemma is almost a tautology.

\begin{lem}\label{blowup}
Fix $K\ge 1$ and $C>0$, suppose that $N$ is a closed manifold, and let $(f_n \colon M \to N)$ be a sequence of $K$-quasiregular maps with $\Vert\mu_{f_n}\Vert\le C$ for all $n$. Assume that the maps $f_n$ converge locally uniformly in the complement of a finite set of poles $P\subset M$ to a map $f\colon M\to N$. Suppose also that, for each $p\in P$, we have a renormalizing sequence $(\sigma^p_n)$ for $(f_n)$ at $p$, and let $\widehat F^p:\widehat X^p\to N$ be the one-point compactification of the associated blow-up.

Then the sequence $(f_n)$ converges tightly to a $K$-quasiregular map $F\colon X\to N$ with $X$ and $F$ as in \eqref{pekkaisapain} and in \eqref{noooojustkidding}, respectively.
\end{lem}

\begin{proof}
For each $p\in P$, let $(\pi^p_n)$ be an asymptotically conformal pinching sequence along which $(f_n \circ \sigma^p_n)$ conveges to $F^p$ and consider the sequence of maps $(\Pi^p_n:M\to X)$ defined as 
\[
x \mapsto \begin{cases} x &\text{ if } d_M(x,p)>2r_n^p,\\
\exp_{p}\left(2\cdot(d_M(x,p)-r_n^p)\cdot\exp_{p}^{-1}(x)\right) &\text{ if } 2r_n^p\ge d_M(x,p)\ge r_n^p,\\
(\pi_n\circ(\sigma^p_n)^{-1})(x) &\text{ if } r^p_n>d_M(x,p),
\end{cases}
\]
where $r^p_n>0$ is the radius of the image of $\sigma^p_n: \BR^d \to B^M(p,r^p_n)$. 

Since $r_n^p \to 0$ as $n\to \infty$ for each $p\in P$, we may assume that, for each $n\in \BN$, balls $B^M(p,2r_n^p)$ for $p\in P$ are mutually disjoint. 
Let $X$ be the nodal manifold in \eqref{pekkaisapain} and $F:X\to N$ the map in  \eqref{noooojustkidding}. Then, for each $n\in \BN$, we may define $\Pi_n \colon M\to X$ to be in $B^M(p,r_n^p)$ the map $\Pi^p_n$ and to agree with all the maps $\Pi_n^p$ in the complement of all of these balls. Clearly, $(\Pi_n)$ is a sequence of homotopic pinching maps $M \to X$. They are also isometric outside of a fixed compact set, and a fortiori uniformly proper.

Since, for each $p\in P$, the maps $(f_n \circ \sigma^p_n)$ converge tightly to $F^p$ along $(\pi^p_n)$ and $\diam((f_n\circ \Pi_n^{-1})(B^M(p,2r_n^p)\setminus B^M(p,r_n^p))) \to 0$ as $n\to \infty$ by Proposition \ref{lem-neck2}, we observe that the maps $f_n$ converge tightly to $F$ along $(\Pi_n)$.
\end{proof}

\section{The bubbling theorem}
\label{sec:proof-main}

In this section we prove Theorem \ref{main}. 

\begin{named}{Theorem \ref{main}}
Let $M$ and $N$ be closed and oriented Riemannian $d$-manifolds for $d\ge 2$, and let $K\ge 1$ and $D\ge 1$. Then each sequence $(f_n:M\to N)$ of $K$-quasiregular maps of degree $D$ has a subsequence $(f_{n_k})$, which converges tightly to a $K$-quasiregular map $f:X\to N$ on some smooth pure nodal Riemannian $d$-manifold $X$ satisfying 
\begin{equation*}
\tag{\ref{eq:degree}}
D=\sum_{V\in\Strata(X)}\deg(f\vert_V).
\end{equation*}
Furthermore, the main stratum of $X$ is diffeomorphic to $M$ and all bubbles are standard spheres $\BS^d$.
\end{named}

The idea of the proof is to make the bubbles appear by scaling up small balls around the poles of the sequence $(f_n)$. That is, proving that suitably chosen blow-up sequences are actually renormalisation sequences. In some sense the hardest problem is to determine the right scale. We are going to scale so that the obtained maps are {\em normalized} in the sense of the following definition:

\begin{defi*}[Normalized sequence]
A sequence $(f_n \colon \BR^d \to N)$ of maps is \emph{$(K,\omega, D)$-normalized for $K\ge 1$, $\omega>0$, and $D\in \BN$}, if there is a sequence $(\Lambda_n)$ tending to infinity such that 
\begin{itemize}
\item $f_n$ is $K$-quasiregular on $\BB(0,\Lambda_n)$,
\item $\mu_{f_n}(\BB(0,\Lambda_n)\setminus\BB(0,1))\le 2\omega$, and 
\item $\mu_{f_n}(\BB(x,\frac 14))\le \mu_f(\BB(0,\Lambda_n))-\omega$ for all $x\in \BB(0,\Lambda_n)$,
\end{itemize}
for each $n$, and with $\lim_{n\to \infty} \mu_{f_n}(\BB(0,\Lambda_n)) = D\cdot \vol(N)$.
\end{defi*}

\begin{bem}
It follows from Proposition \ref{meat} that, as long as $\omega<\frac 13\vol(N)$, the set of poles of any normalized sequence is contained in the unit ball $\overline{\BB(0,1)}$.
\end{bem}

The following proposition, the main step in the proof of Theorem \ref{main}, is in its essence a local version of the whole theorem.

\begin{prop}[Local bubbling theorem]
\label{prop:lbt}
Let $N$ be a closed Riemannian manifold, $K\ge 1$, $0<\omega<\max\{\frac 1{10}\vol(N),\frac 18\}$, $D\in \BN$, and $(f_n \colon \BR^d \to N)$ a $(K,\omega,D)$-normalized sequence which converges locally uniformly to some $f:\BR^d\setminus P\to N$ in the complement of a set of poles $P\subset\overline{\BB(0,1)}$. Then, up to possible passing to a subsequence, the sequence $(f_n)$ admits a tight renormalisation sequence $(\sigma_n^p)$ at each $p\in P$.
\end{prop}

The first step in the proof of Proposition \ref{prop:lbt} is to apply Proposition \ref{meat} to the given sequence $(f_n)$. Given that we already have notation for the set of poles $P$ and for the limiting map $f:\BR^d\to N$, we can slightly streamline the statement of the latter:

\begin{fact}
\label{fact:lbt-1}
With notation as in Proposition \ref{prop:lbt} we have, up to possibly passing to a subsequence, that for each $p\in P$ there is an integer $d_p\in\BN$ with $d_p\ge 1$ such that the measures $\mu_{f_n}$ converge in the weak-$\ast$ topology to the measure 
\[
\mu=\mu_f+\sum_{p\in P}d_p\cdot\vol(N)\cdot\delta_p,
\]
where $\delta_p$ is the Dirac probability measure centred at $p$. Moreover, 
$$\left(D-2/3\right)\vol(N)<\mu(\BR^d)\le D\cdot\vol(N),$$
and $\vert P\vert\ge 2$ if $f$ is constant.\qed
\end{fact}

Note that Fact \ref{fact:lbt-1} implies that the map $f$ has finite energy. Thus, by (QR8), there is an integer $d\ge 0$ with 
\begin{equation}
\label{eq:lbt-1}
\mu_f(\BR^d)=d\cdot\vol(N).
\end{equation}
Also, $d$ is positive unless $f$ is constant. Now, from Fact \ref{fact:lbt-1} and from \eqref{eq:lbt-1}, we get that 
$D-2/3\le d+\sum_p d_p\le D$. Since $D$, $d$ and the $d_p$'s are all integers, we get that actually
\begin{equation}\label{eq0lyon2}
d+\sum_{p\in P}d_p=D.
\end{equation}
A crucial point in the proof of Proposition \ref{prop:lbt} is that the numbers $d_p$ provided by Fact \ref{fact:lbt-1} are smaller than $D$.
 
\begin{lem}\label{lem-sillycalculation}
We have $d_p < D$ for each $p\in P$.
\end{lem}

\begin{proof}
Suppose first that $D=1$. We claim that $d=1$ as well, and note that together with \eqref{eq0lyon2} this implies that $P=\emptyset$. Well, otherwise we have, again by \eqref{eq0lyon2}, that $d=0$ because all the $d_p$'s are positive. In this case the map $f$ is constant and, by the final claim of the fact, we have $\vert P\vert\ge 2$ and thus $\sum_{p\in P}d_p\ge 2$. This contradicts \eqref{eq0lyon2}. 

Suppose now that $D>1$. If $d\ge 1$ then it follows from \eqref{eq0lyon2} that $\delta_p<D$ because all the members of the sum are positive. If $d=0$ we have that $f$ is constant. We thus get from Fact \ref{fact:lbt-1} that $\vert P\vert\ge 2$ and thus that
\[
D=\sum_{p\in P}d_p \ge d_p+\vert P\vert-1 \ge d_p+1.
\]
We have proved the claim.
\end{proof}

We are now ready to prove Proposition \ref{prop:lbt}.

\begin{proof}[Proof of Proposition \ref{prop:lbt}]
The proof is by induction in the parameter $D\in \BN$. In the base case of induction, that is when $D=1$, we have nothing to prove because from Lemma \ref{lem-sillycalculation} we get that $P=\emptyset$. We might thus assume that Proposition \ref{prop:lbt} holds for all $D'\le D-1$.

Since the claim of the proposition involves all poles of the sequence $(f_n)$, but each pole individually, we might fix from now on a single $p\in P$. 

Now, for any decreasing sequence $\epsilon_i\searrow 0$ of positive numbers we have
$$\lim_{i\to\infty}\lim_{n\to\infty}\mu_{f_n}(\BB(p,\epsilon_i))=\mu(p)=d_p\cdot\vol(N).$$
Note that, while keeping the sequence $(\epsilon_n)$ fixed, we might pass to a subsequence of $(f_n)$ so that
$$\lim_{n\to\infty}\mu_{f_{n}}(\BB(p,\epsilon_n^2))=\lim_{n\to\infty}\mu_{f_{n}}(\BB(p,\epsilon_{n}))=\mu(p)=d_p\cdot\vol(N)$$
and
\[
\mu_{f_{n}}(\BB(p,\epsilon_n^2))\ge (1-\epsilon_n)\mu_{f_{n}}(\BB(p,\epsilon_{n}))
\]
for all $n\in \BN$.

Consider now the affine conformal map
\[
\psi_n:\BB(0,1)\to\BB(p,\epsilon_n),\quad x\mapsto \epsilon_n\cdot x+p,
\]
and note that, for all large enough $n$, the maps
\[
f_n\circ\psi_n:\BB(0,1)\to N
\]
satisfy the conditions of Proposition \ref{spread} for $\omega$ and $\epsilon=\epsilon_n$. It follows that there are sequences of positive numbers $\Lambda_n\to\infty$ and affine conformal embeddings
\[
\phi_n:\BB(0,\Lambda_n)\to\BB(0,1)
\]
such that the $K$-quasiregular maps
\[
(f_n\circ\psi_n)\circ\phi_n:\BB(0,\Lambda_n)\to N
\]
satisfy the following conditions:
\begin{enumerate}
\item $\mu_{f_n\circ\psi_n\circ\phi_n}(\BB(0,\Lambda_n)\setminus\BB(0,1))\le 2\omega$,
\item $\mu_{f_n\circ\psi_n\circ\phi_n}(\BB(x,\frac 14))\le \Vert\mu_{f_n\circ\psi_n}\Vert-\omega$ for all $x\in\BB(0,\Lambda_n)$, and 
\item $\Vert \mu_{f_n\circ\psi_n\circ\phi_n}\Vert\ge\Vert\mu_{f_n\circ\psi_n}\Vert-\epsilon_n$.
\end{enumerate}
For each $n\in \BN$, since the conformal map $\phi_n$ is injective, we also have the bound 
\[
\Vert \mu_{f_n\circ\psi_n\circ\phi_n}\Vert\le\Vert \mu_{f_n\circ\psi_n}\Vert=\mu_{f_n}(\BB(p,\epsilon_n)).
\]
Together with the third point above, we thus get that
\begin{equation}\label{pekkahasafiber}
\lim_{n\to \infty}\Vert \mu_{f_n\circ\psi_n\circ\phi_n}\Vert=\lim_{n\to \infty}\Vert\mu_{f_n\circ\phi_n}\Vert=\mu(p)=d_p\cdot\vol(N).
\end{equation}

We are now ready to wrap all of this together. For each $n\in \BN$, set $r_n=2\epsilon_n$ and extend the map
\[
\psi_n\circ\phi_n:\BB(0,\Lambda_n)\to\BB(p,\epsilon_n)
\]
to a diffeomorphism
\[
\sigma_n:\BR^d\to\BB(p,r_n)
\]
and note that \eqref{pekkahasafiber} implies that 
\begin{equation}\label{eq-sanchitoisrunningaround}
\begin{split}
\text{for all }\epsilon,\delta_0\text{ there are }\delta\in(0,\delta_0)\text{ and }R>0\text{ with }\\
\mu_{f_n}\left(B^M(p,\delta)\setminus \sigma_n(\BR^d\setminus\BB(0,R))\right) < \epsilon\text{ for all }n\in \BN.
\end{split}
\end{equation}

We claim that the blow-up sequence $(\sigma_n)$ is a renormalisation sequence for the sequence $(f_n)$ at $p$. We show first that the maps $f_n\circ\sigma_n$ converge tightly to a map on a nodal manifold. Note first that (1) and (2) above, together with \eqref{pekkahasafiber}, imply that the sequence of maps $(f_n\circ\sigma_n)$ is $(K,\omega,d_p)$-normalized; note that the last condition $\lim_{n\to \infty} \mu_{f_n\circ \sigma_n}(\BB(0,\Lambda_n)) = d_p \cdot\vol(N)$ follows from Proposition \ref{meat}. From Lemma \ref{lem-sillycalculation} we get that $d_p<D$. We thus have, by induction, that the sequence of maps $(f_n\circ\sigma_n)$ has a tight renormalisation sequence at each one of its poles. Having a tight renormalisation sequence at each one of its poles, 
it follows then from Lemma \ref{blowup} that there is some nodal manifold $X^p$ such that the sequence $(f_n\circ\sigma_n)$ converges tightly to a map $F^p\colon X^p\to N$.

Note that all the bubbles of $X^p$ are conformal spheres by Lemma \ref{kor-nodalRn-conformal} and thus compact. We thus get that the pinching maps provided in the proof of Lemma \ref{blowup} are isometric outside of some uniform compact set. 

The second condition in the definition of renormalisation sequence follows thus from \eqref{eq-sanchitoisrunningaround}. We have now checked that $(\sigma_n)$ is a renormalisation sequence for $(f_n)$ at $p$, as we wanted to prove.
\end{proof}

We are now ready to prove Theorem \ref{main}.

\begin{proof}[Proof of Theorem \ref{main}]
The proof of Theorem \ref{main} has many points in common with the proof of the local bubbling theorem. First note that, by Proposition \ref{meat2}, we may assume, up to passing to a subsequence, that there is a finite set $P\subset M$ with the following properties:
\begin{enumerate}
\item The measures $\mu_{f_n}$ converge in the weak-$\ast$ topology to a measure $\mu$ with total measure $\mu(M)=D\cdot\vol(N)$.
\item There is a $K$-quasiregular map $f:M\to N$ with $\mu_f=\mu$ on $M\setminus P$. 
\item The sequence $(f_n)$ converges locally uniformly to $f$ in $M\setminus P$.
\item For each $p\in P$ we have $\mu(p)=d_p\cdot\vol(N)$ for some integer $d_p\ge 1$.
\end{enumerate}
We obtain the desired nodal manifold $X$ by blowing up each one of the poles $p\in P$. The argument is almost the same as in the induction step in the proof of Proposition \ref{prop:lbt}.

\begin{claim}
For each pole $p\in P$, there exists a tight renormalizing sequence $(\sigma^p_n \colon \BR^d \to B^M(p,r^p_n))$ for $(f_n)$ such that the sequence $(f_n \circ \sigma^p_n)$ is $(K,\omega,d_p)$-normalized.
\end{claim}
\begin{proof}[Proof of the claim]
Fix $p\in P$, and choose $\omega$ much smaller than $\frac 18$, than $\frac 13\vol(N)$, and than $\inj(M)$. Let $\epsilon_n\searrow 0$ be a sequence of positive numbers, and replace the given sequence of maps $(f_n)$ by a suitable subsequence satisfying 
\[
\lim_{n\to\infty}\mu_{f_n}(B^M(p,\epsilon_n))=\lim_{n\to\infty}\mu_{f_n}(B^M(p,\epsilon_n^2))=\mu(p).
\]
We identify $\BR^d=T_pM$ via a linear isometry and consider the diffeomorphism
\[
\psi_n:\BB(0,1)\to B^M(p,\epsilon_n),\ \ 
x\mapsto \exp_p(\epsilon_n\cdot x),
\]
where $\exp_p$ is the Riemannian exponential map $T_pM \to M$.

The fact that $\exp_p$ is an infinitesimal isometry at $p$ implies that each $\psi_n$ is $K_n$-quasiconformal for some sequence $K_n\searrow 1$. Thus there exists $K'\ge 1$ for which each $f_n \circ \psi_n:\BB(0,1)\to N$ is $K'$-quasiregular for each $n$, and that the maps $f_n\circ\psi_n$ satisfy the conditions of Proposition \ref{spread} for $2\omega$ and for $\epsilon=2\epsilon_n$. At this point we may repeat word by word the argument in the proof of Proposition \ref{prop:lbt} and get that the maps $\psi_n$ extend to diffeomorphisms
\[
\sigma_n:\BR^d\to B(p,2\epsilon_n)
\]
such that the sequence $(\sigma_n)$ is a renormalisation sequence and that $(f\circ \sigma_n)$ is $(K,\omega,d_p)$-normalized. We have proved the claim.
\end{proof}

As in the end of the induction step in Proposition \ref{prop:lbt}, once we know that the sequence $(f_n)$ has a renormalisation sequence at each pole, we get from Lemma \ref{blowup} that the maps $f_n$ converge tightly to a quasiregular map $f:X\to N$ on some nodal manifold, all of whose bubbles are spheres. Tightness of the convergence, together with Lemma \ref{lem-homotopy}, implies \eqref{eq:degree}. We have proved Theorem \ref{main}.
\end{proof}

\section{The applications}
\label{sec:bubbling-applications}

Armed with Theorem \ref{main}, we explain now the proof of the other theorems mentioned in the introduction. Theorem \ref{sat-finite homotopy classes} follows directly from Theorem \ref{main} and from Lemma \ref{lem-homotopy} above:

\begin{named}{Theorem \ref{sat-finite homotopy classes}}
Let $M$ and $N$ be compact smooth oriented Riemannian $d$-manifolds for some $d\ge 2$. For every $K\ge 1$ and $D\ge 1$ there are only finitely many homotopy classes of degree $D$ maps $M\to N$ which admit a $K$-quasiregular representative.\qed
\end{named}

The proof of Theorem \ref{sat-compact} needs a bit more of a discussion.

\begin{named}{Theorem \ref{sat-compact}}
Let $M$ and $N$ be closed and oriented $d$-dimensional Riemannian manifolds for $d\ge 2$. Then the following statements are equivalent:
\begin{enumerate}
\item The space $\QR_{K,D}(M,N)$ is not compact for some $K\ge 1$ and $D\ge 1$.
\item There exists a non-constant quasiregular map $\BS^d\to N$.
\end{enumerate}
Furthermore, if $\QR_{K,D}(M,N)$ is non-compact, then the map in (2) has degree at most $D$. 
\end{named}

Recall that $\QR_{K,D}(M,N)$ is the set of all $K$-quasiregular maps $f:M\to N$ of degree $\deg(f)=D$, endowed with the topology of uniform convergence.

\begin{proof}
Suppose first that for some $K\ge 1$ and $D\ge 1$ the space $\QR_{K,D}(M,N)$ is not compact and let $f_n:M\to N$ be a sequence in that space which does not have any convergent subsequence therein. Theorem \ref{main} implies however that $(f_n)$ has a subsequence which converges to a quasiregular map 
$$f:X\to N$$
on a nodal manifold $X$ which has a stratum diffeomorphic to $M$ and where all other strata are standard spheres $\BS^n$. The fact that $(f_n)$ has no convergent subsequence in $\QR_{K,D}(M,N)$ implies that the map $f:X\to N$ is non-constant on at least one of the spherical strata of $X$. We have thus found our non-constant quasiregular map $\BS^d\to N$. The upper bound on the degree of this map follows directly from \eqref{eq:degree}.

Now, assume that there is a non-constant $K_1$-quasiregular map $\phi:\BS^d\to N$ for some $K_1\ge 1$. It is due to Alexander \cite{Alexander} that there is then also a piecewise linear, and thus $K_2$-quasiregular for some $K_2\ge 1$, branched cover $\psi:M\to\BS^d$. Identifying $\BS^d$ with $\BR^d\cup\{\infty\}$ consider the sequence of maps
\[
f_n:M\to N,\ x\mapsto \phi(n\cdot\psi(x)).
\]
This is a sequence of $K$-quasiregular maps where $K=K_1K_2$. Moreover, all of them have degree $\deg(\phi)\cdot\deg(\psi)$. Finally, this sequence is not equicontinuous, and this clearly implies that it does not have any convergent subsequence. We have proved Theorem \ref{sat-compact}.
\end{proof}

We are left with Corollary \ref{kor-bubble-rational}.

\begin{named}{Corollary \ref{kor-bubble-rational}}
Let $M$ and $N$ be closed and oriented Riemannian $d$-manifolds and suppose that $\QR_{K,D}(M,N)$ is not compact for some $K\ge 1$ and $D\ge 1$. Then $N$ is a rational homology sphere with a finite fundamental group and non-trivial $\pi_d(N)$. If moreover $d\le 4$ then $N$ is covered by $\BS^d$.
\end{named}
\begin{proof}
Since $\QR_{K,D}(M,N)$ is not compact, we obtain from Theorem \ref{sat-compact} a non-constant quasiregular map $f:\BS^d\to N$. Since $f$ is non-constant, it has non-zero energy and hence positive degree by (QR6). Now, having non-zero degree, $f$ is not homotopically trivial. We have found our non-zero element in $\pi_d(N)$. Also, being of positive degree, the map $f$ induces an injection in rational cohomology, and this implies that $N$ is a rational homology sphere. To see that $N$ has finite $\pi_1$, observe that $f:\BS^d\to N$ lifts to a positive degree map $\tilde f:\BS^d\to\tilde N$ to the universal cover of $N$. Positivity of the degree implies that $\tilde f$ is surjective and thus that $\tilde N$ is compact. This proves that $\pi_1(N)$ is finite.

We consider now what happens in small dimensions. Well, the classification theorem of surfaces implies that the sphere $\BS^2$ is the only oriented surface with finite fundamental group. In dimension $3$ we use in turn the validity of the 3-dimensional Poincar\'e conjecture \cite{Lott} to conclude that $N$, having finite fundamental group, is covered by $\BS^3$. Passing to dimension $4$ note that the argument we used above implies that also the universal cover of $N$ is dominated by $\BS^d$ and thus that it is a rational homology sphere. Now we get from Freedman's classification of simply connected 4-manifolds \cite{Freedman-book} that a simply connected rational homology sphere of dimension $4$ is homeomorphic to $\BS^4$. We are done.
\end{proof}

\begin{bem}
The argument we just gave to end the proof of Corollary \ref{kor-bubble-rational} yields a positive answer to Question \ref{quest1} for dimension $d\le 4$ but does not work for dimension $d\ge 5$. In fact, any simply connected rational homology sphere is dominated by the sphere (c.f. with 2.8 (2) in \cite{delaHarpe}) and already in dimension $5$ there are simply connected rational homology spheres which are not spheres; see \cite{Barden1965} for a classification of all simply connected 5-manifolds. 
\end{bem}

We conclude by explaining how a positive answer to Question \ref{quest1} in the introduction yields, in dimension 3, a proof of the Poincar\'e conjecture. Suppose that $X$ is a homotopy sphere. Any homotopy equivalence $\BS^3\to X$ is obviously $\pi_1$-surjective and has degree $1$. Precomposing with a self-map of $\BS^3$ it follows that there is also a $\pi_1$-surjective map $f:\BS^3\to X$ of degree $3$. Now, a result of Edmonds \cite{Edmonds-deg3} implies that $f$ is homotopic to a PL-branched cover, and thus to a quasiregular map. From a positive answer to Question \ref{quest1} one gets that $X$ is homeomorphic to $\BS^3$.

\bibliographystyle{abbrv}

\end{document}